\documentclass[12pt,a4paper]{article}
\usepackage{amsfonts}
\usepackage{amsfonts}
\usepackage{amsfonts}
\usepackage{amsfonts}
\usepackage{amsfonts}
\usepackage{mathrsfs}
\usepackage{amsfonts}
\usepackage{mathrsfs}
\usepackage{amsfonts}
\usepackage{amsfonts}
\usepackage{mathrsfs}
\usepackage{mathrsfs}
\usepackage{amsfonts}
\usepackage{amsfonts}
\usepackage{amsfonts}
\usepackage{amsfonts}
\usepackage{mathrsfs}

\usepackage{amsfonts}
\usepackage{mathrsfs}
\usepackage{mathrsfs}
\usepackage{mathrsfs}
\usepackage{mathrsfs}
\usepackage{amsfonts}
\usepackage{mathrsfs}
\usepackage{mathrsfs}
\usepackage{amsfonts}
\usepackage{amsfonts}
\usepackage{amsfonts}
\usepackage{amsfonts}
\usepackage{amsfonts}
\usepackage{amsfonts}
\usepackage{amsfonts}
\usepackage{amsfonts}
\usepackage{amsfonts}
\usepackage{mathrsfs}
\usepackage{amsfonts}
\usepackage{mathrsfs}
\usepackage{amsfonts}
\usepackage{mathrsfs}
\usepackage{amsfonts}
\usepackage{amsfonts}
\usepackage{amsfonts}
\usepackage{amsfonts}
\usepackage{amsfonts}
\usepackage{amsfonts}
\usepackage{amsfonts}
\usepackage{amsfonts}
\usepackage{amsfonts}
\usepackage{amsfonts}
\usepackage{amsfonts}
\usepackage{amsfonts}
\usepackage{mathrsfs}
\usepackage{mathrsfs}
\usepackage{mathrsfs}
\usepackage{mathrsfs}
\usepackage{mathrsfs}
\usepackage{mathrsfs}
\usepackage{mathrsfs}
\usepackage{mathrsfs}
\usepackage{mathrsfs}
\usepackage{amsfonts}
\usepackage{amsfonts}
\usepackage{amsfonts}
\usepackage{amsfonts}
\usepackage{amsfonts}
\usepackage{amsfonts}
\usepackage{mathrsfs}
\usepackage{amsfonts}
\usepackage{amsfonts}
\usepackage{amsfonts}
\usepackage{amsfonts}
\usepackage{amsfonts}
\usepackage{amsfonts}
\usepackage{amsfonts}
\usepackage{amsfonts}
\usepackage{amsfonts}
\usepackage{amsfonts}
\usepackage{amsfonts}
\usepackage{amsfonts}
\usepackage{mathrsfs}
\usepackage{amsfonts}
\usepackage{mathrsfs}
\usepackage{amsfonts}
\usepackage{amsfonts}
\usepackage{amsfonts}
\usepackage{mathrsfs}
\usepackage{amsfonts}
\usepackage{amsfonts}
\usepackage{amsfonts}
\usepackage{amsfonts}
\usepackage{amsmath}
\usepackage{mathrsfs}
\usepackage{amsfonts}
\usepackage{amsfonts}
\usepackage{amsfonts}
\usepackage{amsfonts}
\usepackage{amsfonts}
\usepackage{amsfonts}
\usepackage{amsfonts}
\usepackage{amsfonts}
\usepackage{amsfonts}
\usepackage{amsfonts}
\usepackage{amsfonts}
\usepackage{mathrsfs}
\usepackage{amsfonts}
\usepackage{amsfonts}
\usepackage{amsfonts}
\usepackage{amsfonts}
\usepackage{amsmath}
\usepackage{mathrsfs}
\usepackage{mathrsfs}
\usepackage{mathrsfs}
\usepackage{mathrsfs}
\usepackage{mathrsfs}
\usepackage{mathrsfs}
\usepackage{mathrsfs}
\usepackage{mathrsfs}
\usepackage{amsfonts}
\usepackage{amsfonts}
\usepackage{amsfonts}
\usepackage{amsfonts}
\usepackage{amsfonts}
\usepackage{amsfonts}
\usepackage{amsfonts}
\usepackage{amsfonts}
\usepackage{amsfonts}
\usepackage{amsfonts}
\usepackage{amsfonts}
\usepackage{amsfonts}
\usepackage{mathrsfs}
\usepackage{amsfonts}
\usepackage{mathrsfs}
\usepackage{amsfonts}
\usepackage{amsfonts}
\usepackage{amsfonts}
\usepackage{amsfonts}
\usepackage{amsfonts}
\usepackage{amsfonts}
\usepackage{amsfonts}
\usepackage{amsfonts}
\usepackage{amsfonts}
\usepackage{amsfonts}
\usepackage{amsfonts}
\usepackage{amsfonts}
\usepackage{amsfonts}
\usepackage{amsfonts}
\usepackage{amsfonts}
\usepackage{amsfonts}
\usepackage{amsfonts}
\usepackage{amsfonts}
\usepackage{mathrsfs}
\usepackage{mathrsfs}
\usepackage{mathrsfs}
\usepackage{mathrsfs}
\usepackage{amsfonts}
\usepackage{amsfonts}
\usepackage{amsfonts}
\usepackage{amsfonts}
\usepackage{amsfonts}
\usepackage{amsfonts}
\usepackage{amsfonts}
\usepackage{amsfonts}
\usepackage{amsfonts}
\usepackage{amsfonts}
\usepackage{amsfonts}
\usepackage{amsfonts}
\usepackage{amsfonts}
\usepackage{amsfonts}
\usepackage{amsfonts}
\usepackage{amsfonts}
\usepackage{amsfonts}
\usepackage{mathrsfs}
\usepackage{mathrsfs}
\usepackage{amsfonts}
\usepackage{amsfonts}
\usepackage{amsfonts}
\usepackage{amsfonts}
\usepackage{amsfonts}
\usepackage{amsfonts}
\usepackage{amsfonts}
\usepackage{amsfonts}
\usepackage{amsfonts}
\usepackage{amsfonts}
\usepackage{amsfonts}
\usepackage{amsfonts}
\usepackage{amsfonts}
\usepackage{mathrsfs}
\usepackage{mathrsfs}
\usepackage{amsfonts}
\usepackage{amsfonts}
\usepackage{amsfonts}
\usepackage{amsfonts}
\usepackage{amsfonts}
\usepackage{amsfonts}
\usepackage{mathrsfs}
\usepackage{mathrsfs}
\usepackage{amsfonts}
\usepackage{amsfonts}
\usepackage{amsfonts}
\usepackage{amsfonts}
\usepackage{amsfonts}
\usepackage{amsfonts}
\usepackage{amsfonts}
\usepackage{amsfonts}
\usepackage{amsfonts}
\usepackage{amsfonts}
\usepackage{amsfonts}
\usepackage{amsfonts}
\usepackage{amsfonts}
\usepackage{amsfonts}
\usepackage{mathrsfs}
\usepackage{amsfonts}
\usepackage{amsfonts}
\usepackage{amsfonts}
\usepackage{amsfonts}
\usepackage{amsfonts}
\usepackage{amsfonts}
\usepackage{amsfonts}
\usepackage{amsfonts}
\usepackage{amsfonts}
\usepackage{amsfonts}
\usepackage{amsfonts}
\usepackage{amsfonts}
\usepackage{amsfonts}
\usepackage{amsfonts}
\usepackage{amsfonts}
\usepackage{amsfonts}
\usepackage{amsmath}
\usepackage{mathrsfs}
\usepackage{amsfonts}
\usepackage{amsfonts}
\usepackage{amsfonts}
\usepackage{amsfonts}
\usepackage{amsmath}
\usepackage{mathrsfs}
\usepackage{amsfonts}
\usepackage{mathrsfs}
\usepackage{mathrsfs}
\usepackage{mathrsfs}
\usepackage{mathrsfs}
\usepackage{amsmath}
\usepackage{mathrsfs}
\usepackage{amsfonts}
\usepackage{color}
\usepackage{mathrsfs}

\usepackage{stmaryrd}
\usepackage{amsmath}
\usepackage{amssymb}

\usepackage{bbm}
\usepackage{mathrsfs}
\usepackage{graphicx}

\usepackage{enumerate}
\usepackage{theorem}

\makeatletter
\def\tank#1{\protected@xdef\@thanks{\@thanks
        \protect\footnotetext[0]{#1}}}
\def\bigfoot{

    \@footnotetext}
\makeatother

\newcommand{\ea}{\end{array}}
\newtheorem{thm}{Theorem}[section]
\newtheorem{prop}{Proposition}[section]

\newtheorem{lem}{Lemma}[section]
\newtheorem{defi}{Definition}[section]
\newtheorem{rmk}{Remark}[section]
\numberwithin{equation}{section}

{\theorembodyfont{\rmfamily}
}

\newenvironment{proof}{Proof.}

\def\esssup {{ \hbox{ ess\ sup} }}

\allowdisplaybreaks

\begin{document}
\title{\Large \bf Stochastic 2D Navier-Stokes equations on time-dependent domains }

\author{{Wei Wang}$^1$\footnote{E-mail:ww12358@mail.ustc.edu.cn}~~~ {Jianliang Zhai}$^1$\footnote{E-mail:zhaijl@ustc.edu.cn}~~~ {Tusheng Zhang}$^{2}$\footnote{E-mail:Tusheng.Zhang@manchester.ac.uk}
\\
 \small  1. Key Laboratory of Wu Wen-Tsun Mathematics CAS, \\
 \small  School of Mathematical Sciences,
 \small  University of Science and Technology of China,\\
 \small  Hefei, Anhui 230026, China.\\
 \small  2. Department of Mathematics, University of Manchester,\\
 \small  Oxford Road, Manchester, M13 9PL, UK.
}
\date{}
\maketitle

\begin{center}
\begin{minipage}{130mm}
{\bf Abstract:}
We establish the existence and uniqueness of solutions to  stochastic 2D Navier-Stokes equations in a time-dependent domain driven by  Brownian motion. A martingale solution is  constructed through domain transformation and appropriate Galerkin approximations on time-dependent spaces. The probabilistic  strong solution follows from the pathwise uniqueness and the Yamada-Watanable theorem.

\vspace{3mm} {\bf Keywords:} Stochastic Navier-Stokes equations; time-dependent domain; tightness; Yamada-Watanabe theorem.

\vspace{3mm} {\bf AMS Subject Classification:}  60H15, 35Q30, 76D05.
\end{minipage}
\end{center}
\begin{section}{Introduction}
Fix $T>0$ and let $\mathcal{O}_T=\bigcup\limits_{t\in[0,T]} \mathcal{D}(t)\times \left\{t\right\}$ be a  non-cylindrical space-time domain, each
 $\mathcal{D}(t),\, t\in[0,T]$, being a bounded open domain in $\mathbb{R}^2$ with smooth boundary $\Gamma(t):=\partial\mathcal{D}(t)$. Let $\hat{\Gamma}_T=\bigcup\limits_{t\in[0,T]} \Gamma(t)\times \left\{t\right\}$ represent the boundary of $\mathcal{O}_T$, and set $\overline{\mathcal{O}}_T=\mathcal{O}_T\bigcup \hat{\Gamma}_T$.

\indent Consider   2D stochastic Navier-Stokes equation with the Dirichlet boundary conditions:
\begin{align}
\mathrm{d}\boldsymbol{u}(t)-\mu \Delta \boldsymbol{u}(t)\mathrm{d}t+\left(\boldsymbol{u}(t)\cdot\nabla\right)\boldsymbol{u}(t)\mathrm{d}t+\nabla p(t)\mathrm{d}t&=\boldsymbol{f}(t)\mathrm{d}t+\boldsymbol{\sigma}(t) \mathrm{d}W(t), &x &\in \mathcal{D}(t),t\in (0,T],\nonumber\\
\mathrm{div}\ \boldsymbol{u}(t)&=0,  & x&\in \mathcal{D}(t),t\in[0,T],\label{add}\\
\boldsymbol{u}(t)&=0,  &x &\in \Gamma(t),t\in[0,T],\nonumber\\
\boldsymbol{u}(x,0)&=\boldsymbol{u}_0(x), & x &\in \mathcal{D}(0),\nonumber
\end{align}
\noindent where $\boldsymbol{u}(x,t)=\left(u^1(x,t),u^2(x,t)\right) :\overline{\mathcal{O}}_T\mapsto \mathbb{R}^2$ and $p(x,t):\overline{\mathcal{O}}_T\mapsto \mathbb{R}$ are the unknown velocity field and pressure of the fluid, respectively; the constant $\mu>0$ is the coefficient of  viscosity, without loss of generality, here we take $\mu=1$; $\boldsymbol{f}(x,t):\overline{\mathcal{O}}_T\mapsto \mathbb{R}^2$ is the  external force and $\boldsymbol{u}_0(x): \mathcal{D}(0)\mapsto \mathbb{R}^2$ is the  initial velocity; $\left\{W(t),t\in [0,T]\right\}$ is a one-dimensional Brownian motion defined on a  stochastic basis $(\Omega,\mathscr{F},(\mathscr{F}_t)_{0\leq t\leq T},\mathbb{P})$, where $(\mathscr{F}_t)_{0\leq t\leq T}$ satisfies the usual conditions. For the precise conditions on $\mathcal{O}_T$, $\boldsymbol{f}$, $\boldsymbol{u}_0$, and $\boldsymbol{\sigma}$, we refer the reader to Section 2 and Section 3.\\
\indent The purpose of this paper is to establish the existence and uniqueness of the  solution  to equation (\ref{add}).

The Navier-Stokes equation is an important  model for atmospheric and ocean dynamics, water flow, and  other viscous flow. It has been extensively  studied by many authors, mainly on  a fixed domain which is independent of time.  For deterministic 2D Navier-Stokes equations, we refer readers to Hopf \cite{EH}, Ladyzhenskaya \cite{OAL}, Leray \cite{JL}, Lions and Prodi \cite{JG}, Temam \cite{RT}, and references therein.

\indent 
Since the seminal work  \cite{BT} by Bensoussan and Temam, a great number of papers have been devoted to the subject of stochastic Navier-Stokes equations. It is still an ongoing very active research area.  We mention some of the relavant works. The well-posedness of stochastic Navier-Stokes equations were studied by Flandoli and Gatarek \cite{FD}, Menaldi and Sritharan \cite{MS}, Liu and R\"{o}ckner \cite{Liu R}. Large deviations for   2D stochastic Navier-Stokes equations were proved by  Chang \cite{MHC}, Sritharan and Sundar \cite{SS}, Wang, Zhai and Zhang \cite{WZZ2015}. Ergodicity of   2D Navier-Stokes equations with stochastic
forcing were studied by Flandoli and Maslowski \cite{FB}, Hairer and Mattingly  \cite{MM}.\\
\indent When modeling  fluid, the area usually changes with time. It is therefore  important to consider Navier-Stokes equations in a time-varying domain. There are a few papers on this topic in the deterministic setting. Applying the penalty methods, Fujita and Sauer \cite{HN}  proved the  existence and uniqueness of the solution to Navier-Stokes equations in domains  with the boundaries covered by  finite number of  $C^3$-class  simple closed surfaces. In Bock \cite{DNB} and Inoue-Wakimoto \cite{IM}, the authors  transformed  Navier-Stokes equations on time-dependent domain into nonlinear partial differential equations on  a cylindrical domain.
The equations in a time-dependent domain with Neumann boundary conditions were considered by Filo and Zau\v{s}kov\'{a} \cite{FZ}. Some related free boundary problems were considered by Bae \cite{HB}, Shibata and Shimizu \cite{YS}.\\
 \indent Taking into account the influence of unknown external forces, which may be caused by environmental noises in both mathematical and physical senses, in this paper we consider   2D stochastic Navier-Stokes equations in a time-dependent domain. To the best of our  knowledge, this is the first result of this kind. The moving boundary of the domain raises serious  difficulties, for example, we can't solve the equation in the usual setting of  a Gelfand triple like in the case of a fixed domain, the It\^{o} formula widely used in the literature is also not available as the state space of the solution changes with time.  Due to the lack of the It\^{o} formula, the  local monotonicity arguments used in the literature for stochastic Navier-Stokes equations fail to work. In this paper, we use finite dimensional approximations. The main difficulty lies in the proof of the tightness of the family of the laws of  the finite dimensional approximations because the state space of the solution changes with time. To overcome this, we need to find a suitable criterion to characterize the compact subsets of the state space of the solution. We  first obtain the existence of the  probabilistic weak solution and then prove the pathwise uniqueness of the solutions. The probabilistic strong solutions are obtained thanks to the Yamada-Watanabe theorem.

\vskip 0.4cm
 \indent The paper is divided into four sections. In Section 2, we  introduce the time-dependent domain and its transformation.  In Section 3,  we  lay out the setup and state the main result. Section 4 is devoted to the proofs of the main results.
\end{section}
\begin{section}{Transformation of domain}
In Euclidean space $\mathbb{R}^n$, we write $\boldsymbol{x}=(x^1,x^2,\cdots,x^n)$ as the points of the space, and $\boldsymbol{x}^{\mathrm T}$ means the transposition of $\boldsymbol{x}$. We will make the following assumptions about the region $\mathcal{O}_T$. Assume that there exists a bounded open domain $\widetilde{\mathcal{D}}\subset \mathbb{R}^2$ satisfying the following condition:
\vskip 0.1cm

(A1) Let $\widetilde{\mathcal{O}}_T=\widetilde{\mathcal{D}}\times [0,T]$, there exists  a level-preserving $C^{\infty}$-diffeomorphism $L$ from  $\overline{\mathcal{O}}_T$(the closure of $\mathcal{O}_T)$ to $\overline{\widetilde{\mathcal{O}}}_T$ (the closure of $ \widetilde{\mathcal{O}}_T$), i.e.,
\begin{align*}
(y,s)=L(x,t)=&(y^1(x,t),y^2(x,t),t),\ \ (x,t)\in \overline{\mathcal{O}}_T,\ (y,s)\in \overline{\widetilde{\mathcal{O}}}_T,
\end{align*}
and
\begin{equation}\label{eq M1}
det\  M(x,t)\equiv J(t)^{-1}>0,\ \ (x,t)\in \overline{\mathcal{O}}_T,
\end{equation}
here
\begin{eqnarray}\label{eq M}
M(x,t)
=
\begin{bmatrix}
\partial y^1(x,t)/\partial x^1&\partial y^2(x,t)/\partial x^1\\
\partial y^1(x,t)/\partial x^2&\partial y^2(x,t)/\partial x^2
\end{bmatrix},\ \ (x,t)\in \overline{\mathcal{O}}_T.
\end{eqnarray}

\indent Let $L^{-1}$ be the inverse of $L$, which means that $L^{-1}:\overline{\widetilde{\mathcal{O}}}_T \mapsto \overline{\mathcal{O}}_T$, and for every $(y,s)\in \overline{\widetilde{\mathcal{O}}}_T$ and $(x,t)\in \overline{\mathcal{O}}_T$ with $(y,s)=L(x,t)$,
\begin{equation*}
L^{-1}(y,s)=L^{-1}L(x,t)=(x,t).
\end{equation*}
Note that the time variable remains unchanged during the transformation, i.e., $s(x,t)=t$.
For each function $\Gamma:\overline{\mathcal{O}}_T\mapsto \mathbb{R}^2$,  $\widetilde{\Gamma}$ will always mean a function mapping $\overline{\widetilde{\mathcal{O}}}_T$ into $\mathbb{R}^2$ obtained by the transformation
\begin{equation}\label{eq Tr}
\widetilde{\Gamma}(y,s)=\Gamma(L^{-1}(y,s))M(L^{-1}(y,s)).
\end{equation}
Conversely, we can recover the function  $\Gamma$ from $\widetilde{\Gamma}$ by setting
$$\Gamma(x,t)=\widetilde{\Gamma}(L(x,t))M(x,t)^{-1}.$$
In the following, we will set
\begin{equation*}
\boldsymbol{\widetilde{u}}(y,s)=\boldsymbol{u}(L^{-1}(y,s))M(L^{-1}(y,s)),
\end{equation*}
\begin{equation*}
\boldsymbol{\widetilde{u}}_0(y,s)=\boldsymbol{u}_0(L^{-1}(y,s))M(L^{-1}(y,s)),
\end{equation*}
\begin{equation*}
\boldsymbol{\widetilde{f}}(y,s)=\boldsymbol{f}(L^{-1}(y,s))M(L^{-1}(y,s)),
\end{equation*}
\begin{equation*}
\boldsymbol{\widetilde{\sigma}}(y,s)=\boldsymbol{\sigma}(L^{-1}(y,s))M(L^{-1}(y,s)).
\end{equation*}
Note that, under this transformation, we have
\begin{lem}\label{lem 3}
$\mathrm{div}\ \Gamma(x,t)=0$, $\forall (x,t)\in \mathcal{O}_T$ if and only if $\mathrm{div}\ \widetilde{\Gamma}(y,s)=0$, $\forall (y,s)\in \widetilde{\mathcal{O}}_T$.
\end{lem}
\begin{proof}
Indeed,  for each $(y,s)\in \widetilde{\mathcal{O}}_T$ and $(x,t)\in {\mathcal{O}}_T$ such that
$
(y,s)=L(x,t),
$ we recall
\begin{equation*}
\widetilde{\Gamma}(y,s)=\Gamma(x,t)M(x,t).
\end{equation*}
If we set
\begin{eqnarray}\label{eq A}
K
=
\frac{\partial(x^1,x^2)}{\partial(y^1,y^2)}
=
\begin{bmatrix}
\partial x^1/\partial y^1&\partial x^1/\partial y^2\\
\partial x^2/\partial y^1&\partial x^2/\partial y^2
\end{bmatrix},
\end{eqnarray}
then
\begin{eqnarray*}
KM^{\mathrm T}
=
\begin{bmatrix}
\sum\limits^{2}_{k=1}(\partial x^1/\partial y^k)(\partial y^k/\partial x^1) &\sum\limits^{2}_{k=1}(\partial x^1/\partial y^k)(\partial y^k/\partial x^2)\\
\sum\limits^{2}_{k=1}(\partial x^2/\partial y^k)(\partial y^k/\partial x^1)
&\sum\limits^{2}_{k=1}(\partial x^2/\partial y^k)(\partial y^k/\partial x^2)
\end{bmatrix}
=
\begin{bmatrix}
\partial x^1/\partial x^1  &\partial x^1/\partial x^2\\
\partial x^2/\partial x^1
&\partial x^2/\partial x^2
\end{bmatrix}=I.
\end{eqnarray*}
This means that   $M^{\mathrm T}$ is the inverse of $K$, i.e.,
\begin{eqnarray}\label{eq Ainver}
M^{\mathrm T}
=
K^{-1}
=
\frac{1}{det K}
\begin{bmatrix}
\partial x^2/\partial y^2&-\partial x^1/\partial y^2\\
-\partial x^2/\partial y^1&\partial x^1/\partial y^1
\end{bmatrix},
\end{eqnarray}
where
\begin{equation*}
det K=\frac{1}{det M}=J(t).
\end{equation*}
The equation (\ref{eq Ainver}) illustrates that
\begin{eqnarray*}
M
=
\left(J(t)^{-1}\right)
\begin{bmatrix}
\partial x^2/\partial y^2&-\partial x^2/\partial y^1\\
-\partial x^1/\partial y^2&\partial x^1/\partial y^1
\end{bmatrix},
\end{eqnarray*}\
namely,
\begin{align}\label{eq MA}
\partial y^1/\partial x^1&=\left(J(t)^{-1}\right) \partial x^2/\partial y^2, \nonumber \\
\partial y^2/\partial x^1&=-\left(J(t)^{-1}\right) \partial x^2/\partial y^1, \nonumber \\
\partial y^1/\partial x^2&=-\left(J(t)^{-1}\right) \partial x^1/\partial y^2,  \\ \nonumber
\partial y^2/\partial x^2&=\left(J(t)^{-1}\right) \partial x^1/\partial y^1.
\end{align}
Apply equation (\ref{eq MA}) in the following  calculation to obtain
\begin{align*}
\mathrm{div}\ \widetilde{\Gamma}(y,s)=&\sum_{j=1}^{2}\partial \widetilde{\Gamma}^{j}(y,s)/\partial y^{j}\\=&\sum_{j=1}^{2}\sum_{k=1}^{2}[(\partial (\partial y^j/\partial x^k)/\partial y^j) \Gamma^k(x,t)+(\partial y^j/\partial x^k )(\partial \Gamma^k(x,t)/\partial y^j)]\\
=&\sum_{j=1}^{2}\sum_{k=1}^{2} (\partial y^j/\partial x^k)( \partial \Gamma^k(x,t)/\partial y^j)\\
&+ \left[\Gamma^1(x,t)\left\{\partial\left(J(t)^{-1} (\partial x^2/\partial y^2)\right)/\partial y^1\right\} -\Gamma^1(x,t)\left\{\partial\left(J(t)^{-1} (\partial x^2/\partial y^1)\right)/\partial y^2\right\}\right]\\
&+\left[\Gamma^2(x,t)\left\{\partial\left(J(t)^{-1} (\partial x^1/\partial y^2)\right)/\partial y^1\right\}-\Gamma^2(x,t)\left\{\partial\left(J(t)^{-1} (\partial x^1/\partial y^1)\right)/\partial y^2\right\}\right]\\
=&\sum_{k=1}^{2}  \partial \Gamma^k(x,t)/\partial x^k=\mathrm{div}\ \Gamma(x,t).
\end{align*}
The proof of Lemma \ref{lem 3} is complete.
\end{proof}
\vskip 0.4cm

Set $\widetilde{p}(y,s)=p(L^{-1}(y,s))$. The transformation $\widetilde{\cdot}$ and the above lemma imply that $(\ref{add})$ can be transformed into the following problem on $\widetilde{\mathcal{O}}_T:$
\begin{align}
\mathrm{d}\widetilde{\boldsymbol{u}}(s)-F \widetilde{\boldsymbol{u}}(s)\mathrm{d}s\,+\,&G \widetilde{\boldsymbol{u}}(s)\mathrm{d}s +\mathcal{N}\widetilde{\boldsymbol{u}}(s)\mathrm{d}s+\nabla_h \widetilde{p}(s)\mathrm{d}s \label{traeq} \\
   & =\widetilde{\boldsymbol{f}}(s)\mathrm{d}s+\widetilde{\boldsymbol{\sigma}}(s) dW(s),\ \ \ \ y \in  \widetilde{\mathcal{D}},s\in(0,T],\nonumber\\
\mathrm{div}\ \widetilde{\boldsymbol{u}}(s)&=0, \ \ \ \ \ \ \ \ \ \ \ \ \ \ \ \ \ \ \ \ \ \ \ \ \ \ \ \ \ \ y \in  \widetilde{\mathcal{D}},s\in[0,T],\nonumber\\
\widetilde{\boldsymbol{u}}(s)&=0, \ \ \ \ \ \ \ \ \ \ \ \ \ \ \ \ \ \ \ \ \ \ \ \ \ \ \ \ \ \ y  \in \partial  \widetilde{\mathcal{D}} ,s\in[0,T],\nonumber\\
\widetilde{\boldsymbol{u}}(y,0)&=\widetilde{\boldsymbol{u}}_0(y), \ \ \ \ \ \ \ \ \ \ \ \ \ \ \ \ \ \ \ \ \ \ \ \  y \in  \widetilde{\mathcal{D}},\nonumber
\end{align}
where, for $i=1,2,$
\begin{align*}
(F \widetilde{\boldsymbol{u}})^{i}&=\sum_{j=1}^{2}\sum_{k=1}^{2} h^{jk}\nabla_j \nabla_k \widetilde{u}^i,\\
(G \widetilde{\boldsymbol{u}})^{i}&=\sum_{j=1}^{2}(\partial y^j/\partial t)\nabla_j \widetilde{u}^i+\sum_{j=1}^{2}\sum_{k=1}^{2}(\partial y^i/\partial x^k)(\partial^2 x^k/\partial s\partial y^j) \widetilde{u}^j,\\
(\mathcal{N}\widetilde{\boldsymbol{u}})^i&=\sum_{j=1}^{2} \widetilde{u}^j\nabla_j \widetilde{u}^i,\\
(\nabla_h \widetilde{p})^{i}&=\sum_{j=1}^{2}h^{ij}(\partial \widetilde{p}/\partial y^j),
\end{align*}
and, for $i,j\in \{1,2\}$,
\begin{align*}
h^{ij}&=\sum_{k=1}^{2}(\partial y^i/\partial x^k)(\partial y^j/\partial x^k),\ \  h_{ij}=\sum_{k=1}^{2}(\partial x^k/\partial y^i)(\partial x^k/\partial y^j),\\
\nabla_j \widetilde{u}^i&=\partial \widetilde{u}^i/\partial y^j+\sum_{k=1}^{2}\Phi_{jk}^i \widetilde{u}^k,\\
\nabla_j \nabla_k \widetilde{u}^i &=\partial (\nabla_k\widetilde {u}^i)/\partial y^j+\sum_{l=1}^{2}\Phi_{jl}^{i} \nabla_{k} \widetilde{u}^l-\sum_{l=1}^{2}\Phi_{jk}^{l} \nabla_{l} \widetilde{u}^i,\\
2\Phi_{ij}^{k}&=\sum_{l=1}^{2}h^{kl}(\partial h_{il}/\partial y^j+\partial h_{jl}/\partial y^i-\partial h_{ij}/\partial y^l)\\
  &=2\sum_{l=1}^{2}(\partial y^k/\partial x^l)(\partial^2 x^l/\partial y^j\partial y^i).
\end{align*}
\begin{rmk}
Set $H_1=(h^{ij})_{2\times 2}$ and  $H_2=(h_{ij})_{2\times 2}$. The condition (A1) implies that
 \begin{equation*}
H_1=M^{\mathrm T} M, \quad H_1^{-1}=H_2,\quad det\ H_2=J(t)^2.
\end{equation*}
\end{rmk}
\end{section}
\begin{section}{Statement of the main result}
In this section we introduce the precise definition of solutions and state the main result.
We begin with some notations. Let $C_{0}^{\infty}(\widetilde{\mathcal{D}})$ be the space of all $\mathbb{R}$-valued $C^{\infty}$ functions on $\widetilde{\mathcal{D}}$ with compact supports. We denote by  $\mathbb{L}^2(\widetilde{\mathcal{D}})$ and $\mathbb{H}_0^1(\widetilde{\mathcal{D}})$ the closure of $C_{0}^{\infty}(\widetilde{\mathcal{D}})$ under the following norms:
\begin{equation}
\|\boldsymbol{u}\|_{\mathbb{L}^2(\widetilde{\mathcal{D}})}=\left(\int_{\widetilde{\mathcal{D}}} |\boldsymbol{u}(y)|^2 \mathrm{d} y\right)^{1/2},
\end{equation}
\begin{equation}
\|\boldsymbol{u}\|_{\mathbb{H}_0^1(\widetilde{\mathcal{D}})} =\left(\int_{\widetilde{\mathcal{D}}} |\nabla \boldsymbol{u}(y)|^2 \mathrm{d} y \right)^{1/2}.
\end{equation}
Introduce  the spaces
\begin{align*}
\mathscr{C}&=\{\boldsymbol u\in (C_{0}^{\infty}(\widetilde{\mathcal{D}}))^2:\mathrm{div}\ \boldsymbol{u}=0\},\\
\widetilde{V}&=the\ closure \ of \ \mathscr{C} \ in\  (\mathbb{H}_0^1(\widetilde{\mathcal{D}}))^{2},\\
\widetilde{H}&=the\ closure \ of \ \mathscr{C} \ in\  (\mathbb{L}^2(\widetilde{\mathcal{D}}))^{2}.
\end{align*}
For fixed $t\in [0,T]$, similarly we can define the spaces  $H_t$ and $V_t$ on the domain $\mathcal{D}(t)$:
\begin{align*}
\mathscr{C}_t &=\{\boldsymbol u\in (C_{0}^{\infty}(\mathcal{D}(t)))^2:\mathrm{div}\ \boldsymbol{u}=0\},\\
V_t &=the\ closure \ of \ \mathscr{C}_t \ in\  (\mathbb{H}_0^1(\mathcal{D}(t))^{2},\\
H_t &=the\ closure \ of \ \mathscr{C}_t \ in\  (\mathbb{L}^2(\mathcal{D}(t))^{2}.
\end{align*}

  The corresponding inner products on $H_t$ and $V_t$ are defined as
\begin{equation}\label{eq 3.11}
(\boldsymbol{u},\boldsymbol{v})_t= \int_{\mathcal{D}(t)} \boldsymbol{u}(x)\boldsymbol{v}^{\mathrm {T}}(x)\mathrm{d}x=\sum_{i=1}^{2}\int_{\mathcal{D}(t)} u^{i}(x)v^{i} (x)\mathrm{d}x,
\end{equation}
\begin{equation}
(\nabla \boldsymbol{u},\nabla \boldsymbol{ v})_t= \sum_{i=1}^{2}\sum_{j=1}^{2} \int_{\mathcal{D}(t)}  (\partial u^i/\partial x^j)(\partial v^i/\partial x^j)\mathrm{d}x.
\end{equation}
For every $ t\in[0,T]$,  $\widetilde{H}$ is a Hilbert space with the inner product
\begin{equation}
\left<\boldsymbol{\widetilde{u}},\boldsymbol{\widetilde{v}}\right>_t= \int_{\widetilde{\mathcal{D}}}\boldsymbol{\widetilde{u}}(y) H_2(y,t) \boldsymbol{\widetilde{v}}^{\mathrm {T}}(y) J(t)\mathrm{d}y,\label{norm1}
\end{equation}\
i.e.
\begin{equation}
\left<\boldsymbol{\widetilde{u}},\boldsymbol{\widetilde{v}}\right>_t=\sum_{i=1}^{2}\sum_{j=1}^{2} \int_{\widetilde{\mathcal{D}}} h_{ij}(y,t) \widetilde{u}^{i}(y)\widetilde{v}^{j}(y) J(t) \mathrm{d}y.
\end{equation}
Similarly, we can define the inner product on $\widetilde{V}$  as follows
\begin{equation}\label{Norm V}
\left<\nabla_h \boldsymbol{\widetilde{u}},\nabla_h \boldsymbol{ \widetilde{v}}\right>_t= \sum_{i=1}^{2}\sum_{j=1}^{2}\sum_{k=1}^{2}\sum_{l=1}^{2}  \int_{\widetilde{\mathcal{D}}} h_{ij}(y,t)h^{kl}(y,t) \nabla_k\widetilde{u}^{i}(y) \nabla_l\widetilde{v}^{j}(y) J(t) \mathrm{d}y,\ t\in[0,T].
\end{equation}

\begin{rmk}
$\widetilde{H}$ and $\widetilde{V}$ are usually equipped with the following inner products, respectively:
\begin{equation}
\left<\boldsymbol{\widetilde{u}},\boldsymbol{\widetilde{v}}\right>_{\widetilde{H}}= \int_{\widetilde{\mathcal{D}}}\boldsymbol{\widetilde{u}}(y) \boldsymbol{\widetilde{v}}^{\mathrm {T}}(y) \mathrm{d}y,
\end{equation}
\begin{equation}
\left<\boldsymbol{ \widetilde{u}},\boldsymbol{ \widetilde{v}}\right> _{\widetilde{V}}=\int_{\widetilde{\mathcal{D}}}\nabla\boldsymbol{\widetilde{u}}(y) (\nabla\boldsymbol{\widetilde{v}})^{\mathrm {T}}(y) \mathrm{d}y.
\end{equation}
The condition (A1) implies that the norms induced by the above two inner products  are equivalent to that induced by (\ref{norm1}) and (\ref{Norm V}), respectively.
\end{rmk}


 After a change of variable, we see that
 \begin{eqnarray}\label{eq 1}
 (\boldsymbol{u},\boldsymbol{v})_t=\left<\boldsymbol{\widetilde{u}},\boldsymbol{\widetilde{v}}\right>_t,\ \ \text{and }
 (\nabla \boldsymbol{ u},\nabla \boldsymbol{ v})_t=\left<\nabla_h  \boldsymbol{\widetilde{u}},\nabla_h \boldsymbol{ \widetilde{v}}\right>_t,\ \forall t\in [0,T].
 \end{eqnarray}
The following notations will be used.
\begin{align*}
|\boldsymbol{\widetilde{u}}|_{t}&=\left<\boldsymbol{\widetilde{u}},\boldsymbol{\widetilde{u}}\right>_t^{1/2}, \quad \emph{\text{for each }} \boldsymbol{\widetilde{u}}\in  \widetilde{H};\\
|\nabla_h  \boldsymbol{\widetilde{u}}|_{t}&=\left<\nabla_h \boldsymbol{\widetilde{u}},\nabla_h \boldsymbol{ \widetilde{u}}\right>_t^{1/2}, \quad \emph{\text{for each }} \   \boldsymbol{\widetilde{u}}\in  \widetilde{V};\\
\|\boldsymbol{u}\|_{t}&=\left(\boldsymbol{u},\boldsymbol{u}\right)_t^{1/2}, \quad \emph{\text{for each }}   \boldsymbol{u}\in  H_t;\\
\|\nabla \boldsymbol{ u}\|_{t}&=(\nabla \boldsymbol{ u},\nabla \boldsymbol{ u})_t^{1/2}, \quad \emph{\text{for each }}  \boldsymbol {u}\in  V_t.
\end{align*}

Identifying $(\widetilde{H},\left<\cdot,\cdot\right>_{\widetilde{H}})$ with its dual space $\widetilde{H}^{*}=(\widetilde{H},\left<\cdot,\cdot\right>_{\widetilde{H}})$, we denote by $\widetilde{V}^{*}$  the dual spaces of $\widetilde{V}$. The corresponding norm in $\widetilde{V}^{*}$ is given by
\begin{equation}\label{eq V star}
\|\boldsymbol{\widetilde{u}}\|_{\widetilde{V}^{*}}=\sup_{\left<\boldsymbol{ \widetilde{v}},\boldsymbol{ \widetilde{v}}\right> _{\widetilde{V}} \leq 1,\boldsymbol{\widetilde{v}}\in \widetilde{V}}\left<\boldsymbol{\widetilde{u}},\boldsymbol{\widetilde{v}}\right>_{\widetilde{H}}.
\end{equation}

The spaces $\mathbb{L}^2([0,T];\widetilde{V})$ and $\mathbb{L}^2([0,T];\widetilde{V}^*)$ are defined as
\begin{equation*}
\mathbb{L}^2([0,T];\widetilde{V})=\left\{\boldsymbol{v}=\left\{\boldsymbol{v}(t)\in \widetilde{V}, t\in[0,T]\right\}\big|\int_{0}^{T} \big(\|\boldsymbol{v}(t)\|_{\widetilde{V}}\big)^2\mathrm{d}t<\infty \right\},
\end{equation*}
\begin{equation*}
\mathbb{L}^2([0,T];\widetilde{V}^{*})=\left\{\boldsymbol{v}
=
\left\{\boldsymbol{v}(t)\in \widetilde{V}^{*}, t\in[0,T]\right\}\big|\int_{0}^{T} \big(\|\boldsymbol{v}(t)\|_{\widetilde{V}^{*}}\big)^{2}\mathrm{d}t<\infty \right\}.
\end{equation*}
$C([0,T];\widetilde{H})$ is the space of continuous functions from the closed interval $[0,T]$ to the Hilbert space $(\widetilde{H},\left<\cdot,\cdot\right>_{\widetilde{H}})$.

For each $t\in [0,T]$, identifying $(\widetilde{H},\left<\cdot,\cdot\right>_t)$ with its dual space $\widetilde{H}^{*}_t=(\widetilde{H},\left<\cdot,\cdot\right>_t)$, we denote by $\widetilde{V}_t^{*}$  the dual spaces of $\widetilde{V}$. The corresponding norm in $\widetilde{V}_t^{*}$ is given by
\begin{equation}
|\boldsymbol{\widetilde{u}}|_t^{*}=\sup_{|\nabla_h \boldsymbol{\widetilde{v}}|_{t} \leq 1,\boldsymbol{\widetilde{v}}\in \widetilde{V}}\left<\boldsymbol{\widetilde{u}},\boldsymbol{\widetilde{v}}\right>_{t}.
\end{equation}
Similarly for each $t\in [0,T]$, identifying $(H_t,\left(\cdot,\cdot\right)_t)$ with its dual space $H^{*}_t=(H_t,\left(\cdot,\cdot\right)_t)$, we denote by ${V}_t^{*}$  the dual spaces of $V_t$. The corresponding norm in ${V}_t^{*}$ is given by
\begin{equation}
\|\boldsymbol{u}\|_t^{*}=\sup_{\|\nabla \boldsymbol{ v}\|_{t} \leq 1,\boldsymbol{v}\in V_t} \left(\boldsymbol{u},\boldsymbol{v}\right)_t.
\end{equation}
\end{section}

\noindent For fixed $t\in [0,T]$, we denote by $\mathbb{\pi}_t$  the orthogonal projection from $\mathbb{L}^2 (\mathcal{D}_t;\mathbb{R}^2)$
to $H_t$, and define $A_t$, $B_t$, $b_t$ as follows: for any $\boldsymbol{u},\boldsymbol{v},\boldsymbol{w}\in V_t$,
\begin{equation*}
A_t:V_t\mapsto V_t^{*},\quad A_t\boldsymbol{u}=-\pi_t \Delta\boldsymbol{u};
\end{equation*}
\begin{equation*}
B_t:V_t\times V_t \mapsto V_t^{*},\quad B_t(\boldsymbol{u},\boldsymbol{v})=\pi_t (\boldsymbol{u}\cdot\nabla\boldsymbol{v}),\ \ \ B_t(\boldsymbol{u})=B_t(\boldsymbol{u},\boldsymbol{u});
\end{equation*}
\begin{equation*}
b_t(\boldsymbol{u},\boldsymbol{v},\boldsymbol{w})=\ _{V_t^{*}}\left<B_t(\boldsymbol{u},\boldsymbol{v}),\boldsymbol{w}\right>_{V_t}=\sum_{i,j=1}^{2} \int_{\mathcal{D}(t)}u^i\partial_i v^j w^j \ \mathrm{d}x.
\end{equation*}
We have the estimates for $b_t(\boldsymbol{u},\boldsymbol{v},\boldsymbol{w})$ (see \cite[Lemma 3.4 in Chapter III]{RT}):
\begin{equation}\label{eq b}
|b_t(\boldsymbol{u},\boldsymbol{v},\boldsymbol{w})|\leq C_1 \|\boldsymbol{u}\|_t^{1/2} \|\nabla \boldsymbol{u}\|_{t}^{1/2} \|\nabla \boldsymbol{v}\|_{t} \|\boldsymbol{w}\|_t^{1/2} \|\nabla \boldsymbol{w}\|_{t}^{1/2},\ \boldsymbol{u},\boldsymbol{v},\boldsymbol{w}\in V_t.
\end{equation}
\noindent Taking into account the property: $b_t(\boldsymbol{u},\boldsymbol{v},\boldsymbol{w})=-b_t(\boldsymbol{u},\boldsymbol{w},\boldsymbol{v})$, we have
\begin{equation}
|\ _{V_{t}^{*}}\left<B_t(\boldsymbol{u})-B_t(\boldsymbol{v}),\boldsymbol{u}-\boldsymbol{v}\right>_{V_t}|\leq \frac{1}{2} \|\nabla(\boldsymbol{u}-\boldsymbol{v})\|^{2}_{t}+C_2\| \boldsymbol{u}-\boldsymbol{v}\|^{2}_{t}\|\boldsymbol{v}\|_{\mathbb{L}^4(\mathcal{D}(t);\mathbb{R}^2)}^4,
\end{equation}
\begin{equation}
\ _{V_{t}^{*}}\left<A_t\boldsymbol{u}-A_t\boldsymbol{v}+B_t(\boldsymbol{u})-B_t(\boldsymbol{v}),\boldsymbol{u}-\boldsymbol{v}\right>_{V_t}+C_2\| \boldsymbol{u}-\boldsymbol{v}\|^{2}_{t}\|\boldsymbol{v}\|_{\mathbb{L}^4(\mathcal{D}(t))}^4\geq \frac{1}{2} \|\nabla(\boldsymbol{u}-\boldsymbol{v})\|^{2}_{t}.\label{monotonicity}
\end{equation}
Similarly, we denote
\begin{equation*}
\mathcal{N}(\boldsymbol{\widetilde{u}},\boldsymbol{\widetilde{v}})= \sum_{j=1}^{2} \widetilde{u}^j\nabla_j \widetilde{v}, \ \
\mathcal{N}(\boldsymbol{\widetilde{u}})=\mathcal{N}(\boldsymbol{\widetilde{u}},\boldsymbol{\widetilde{u}}),
\end{equation*}
where $\nabla_j $ and $\mathcal{N}(\boldsymbol{\widetilde{u}})$ are  consistent with the notations in $(\ref{traeq})$.

The spaces $\mathbb{L}^2([0,T];V_{t}^{*})$, $\mathbb{L}^p([0,T];H_{t})$, and
 $\mathbb{L}^{\infty} ([0,T];H_t)$ are defined as
\begin{equation*}
\mathbb{L}^2([0,T];V_{t}^{*})=\left\{\boldsymbol{v}=\left\{\boldsymbol{v}(t)\in V_{t}^{*}, t\in[0,T]\right\}\big|\int_{0}^{T} \big(\|\boldsymbol{v}(t)\|_t^{*}\big)^2\mathrm{d}t<\infty \right\},
\end{equation*}
\begin{equation*}
\mathbb{L}^p([0,T];H_{t})=\left\{\boldsymbol{v}
=
\left\{\boldsymbol{v}(t)\in H_{t}, t\in[0,T]\right\}\big|\int_{0}^{T} \big(\|\boldsymbol{v}(t)\|_t\big)^{p}\mathrm{d}t<\infty \right\},
\end{equation*}
\begin{equation*}
\mathbb{L}^{\infty} ([0,T];H_t)=\Big\{\boldsymbol{v}=\left\{\boldsymbol{v}(t)\in H_{t}, t\in[0,T]\right\}\big| \esssup_{t\in [0,T]} \ \|\boldsymbol{v}(t)\|_t<\infty \Big\}.
\end{equation*}


Now  we define the  solution of  (\ref{add}).
\begin{defi}\label{def solution}
Let $p>2$. For  $\boldsymbol{u}_0\in H_0,\boldsymbol{f}\in \mathbb{L}^2([0,T];V_{t}^{*})$ and $\boldsymbol{\sigma}\in \mathbb{L}^p([0,T];H_{t})$, we call a stochastic process $\boldsymbol{u}$ a solution of (\ref{add}), if
\begin{description}
\item[(i)] $\boldsymbol{u}\in \mathbb{L}^{\infty} ([0,T];H_t)\cap \mathbb{L}^2 ([0,T];V_t),\ \mathbb{P}\texttt{-}a.s.;$
\item[(ii)] $\boldsymbol{u}$ is $(\mathscr{F}_t)_{0\leq t\leq T}$-adapted;
\item[(iii)] for all $\boldsymbol{\varphi}\in \mathscr {C}_{\sigma}(\mathcal{O}_T)= \{\boldsymbol{v}\in (C_0^\infty(\mathcal{O}_{T}))^2\,|\, \mathrm{div}\ \boldsymbol{v} =0$, $\boldsymbol{v} (T)=0$\},
\begin{align}
&-\int_{0}^{T}\left(\boldsymbol{u}(s),\boldsymbol{\varphi}^{\prime}(s)\right)_s\mathrm{d} s +\int_{0}^{T}\left(\nabla\boldsymbol{u} (s), \nabla\boldsymbol{\varphi}(s)\right)_s\mathrm{d}s-\int_{0}^{T}\left((\boldsymbol{u}(s)\cdot\nabla)\boldsymbol{\varphi}(s),  \boldsymbol{u}(s)\right)_s\mathrm{d}s\nonumber\\
&\ =\left(\boldsymbol{u}_0,\boldsymbol{\varphi}(0)\right)_0+\int_{0}^{T}\left(\boldsymbol{f}(s),\boldsymbol{\varphi}(s)\right)_s\mathrm{d}s+\int_{0}^{T}\left(\boldsymbol{\sigma}(s),\boldsymbol{\varphi}(s)\right)_s\mathrm{d}W(s).\label{solution}
\end{align}
\end{description}
\end{defi}
\begin{rmk}\label{Rem 3.2} Set $\mathscr {C}_{\sigma}(\widetilde{\mathcal{O}}_T)= \{\widetilde{\boldsymbol{\chi}}\in (C_0^\infty(\widetilde{\mathcal{O}}_{T}))^2\,|\,   \mathrm{div}\  \widetilde{\boldsymbol{\chi}} =0,\ \widetilde{\boldsymbol{\chi}} (T)=0\}$. Let
$$
\mathscr{C}_1=\{\widetilde{\boldsymbol{v}}\,|\,\boldsymbol{v}\in \mathscr {C}_{\sigma}(\mathcal{O}_T),\ \widetilde{\boldsymbol{v}}(y,s)=\boldsymbol{v}(L^{-1}(y,s))M(L^{-1}(y,s)),\ (y,s)\in \widetilde{\mathcal{O}}_T\},
$$
 and
$$
\mathscr{C}_2=\{\boldsymbol{\chi}\,|\,\widetilde{\boldsymbol{\chi}}\in \mathscr {C}_{\sigma}(\widetilde{\mathcal{O}}_T),\ \boldsymbol{\chi}(x,t)=\widetilde{\boldsymbol{\chi}}(L(x,t))M(x,t)^{-1},\ (x,t)\in \mathcal{O}_T\}.
$$
The condition (A1), (\ref{eq M1}) and Lemma \ref{lem 3} imply that
$$
\mathscr{C}_1=\mathscr {C}_{\sigma}(\widetilde{\mathcal{O}}_T) \text{ and } \mathscr{C}_2=\mathscr {C}_{\sigma}({\mathcal{O}}_T).
$$
The above fact implies that (\ref{solution}) is equivalent to  that
 $\boldsymbol{\widetilde{u}}(y,s)=\boldsymbol{u}(L^{-1}(y,s))M(L^{-1}(y,s)),\\~ (y,s)\in \widetilde{\mathcal{O}}_T$ satisfies the following identity: for any $\widetilde{\boldsymbol{\varphi}}\in \mathscr {C}_{\sigma}(\widetilde{\mathcal{O}}_T)$,
\begin{align}
&-\int_{0}^{T}\!\!\!\! \left<\widetilde{\boldsymbol{u}}(s),\widetilde{\boldsymbol{\varphi}}^{\prime}(s)\right>_s \!\mathrm{d}s
\!-\!\int_{0}^{T}\!\!\!\! \left<\widetilde{\boldsymbol{u}}(s),G\widetilde{\boldsymbol{\varphi}}(s)\right>_s \!\mathrm{d}s
\!+\!\int_{0}^{T}\!\!\!\! \left<\nabla_h \widetilde{\boldsymbol{u}}(s),\nabla_h\widetilde{\boldsymbol{\varphi}}(s)\right>_s \!\mathrm{d}s
\!+\!\int_{0}^{T}\!\!\!\! \left<\mathcal{N}(\widetilde{\boldsymbol{u}}(s)),\widetilde{\boldsymbol{\varphi}}(s)\right>_s\!\mathrm{d}s\nonumber \\
&=\left<\widetilde{\boldsymbol{u}}_0,\widetilde{\boldsymbol{\varphi}}(0)\right>_0+\int_{0}^{T} \left<\widetilde{\boldsymbol{f}}(s),\widetilde{\boldsymbol{\varphi}}(s)\right>_s \mathrm{d}s+\int_{0}^{T} \left<\widetilde{\boldsymbol{\sigma}}(s),\widetilde{\boldsymbol{\varphi}}(s)\right>_s \mathrm{d}W(s).
\end{align}

\end{rmk}
Now we are in a position to state the main result.
\begin{thm}\label{thm1}
Let $p>2$. For $\boldsymbol{u}_0\in H_0,\boldsymbol{f}\in \mathbb{L}^2([0,T];V_{t}^{*})$ and $\boldsymbol{\sigma}\in \mathbb{L}^p([0,T];H_{t})$,
there exists a unique solution to the stochastic Navier-Stokes equation  (\ref{add}).
\end{thm}
\begin{section}{Proof of Theorem \ref{thm1}}
We will use  Galerkin approximations to prove Theorem \ref{thm1}.
Let $\{\widetilde{\boldsymbol{e}}_n\}_{n\in\mathbb{N}}$ be  a sequence of linearly independent elements  in the space $\mathscr{C}$  which is total in $\widetilde{V}$. For $s\in [0, T]$,  let $\{\widetilde{\boldsymbol{w}}_n(s)\}_{n\in\mathbb{N}}$ be its Schmidt orthogonalization with respect to the inner product (\ref{norm1}).
Note that  
\begin{eqnarray}\label{eq w}
\{\widetilde{\boldsymbol{w}}_n(s)\}_{n\in\mathbb{N},s\in[0,T]}
\text{ is smooth in }(y,s), \text{ and }\forall i,j\in \mathbb{N},
\langle\widetilde{\boldsymbol{w}}_i(s),\widetilde{\boldsymbol{w}}_j(s)\rangle_s=\delta_{i,j}, \ s\in[0,T].
\end{eqnarray}

The following result is taken from Lemma 2.7 in \cite{TY}.
\begin{lem}\label{lem 5}
If $\widetilde{\boldsymbol{w}}\in \mathbb{L}^2([0,T];\widetilde{V})$ and $\widetilde{\boldsymbol{w}}'\in \mathbb{L}^2([0,T];\widetilde{V}^*)$, then
$\widetilde{\boldsymbol{w}}\in C([0,T];\widetilde{H})$ and
\begin{eqnarray}\label{eq lem 5 1}
\frac{d|\widetilde{\boldsymbol{w}}(s)|_s^2}{ds}
=
2\langle\widetilde{\boldsymbol{w}}'(s)+G\widetilde{\boldsymbol{w}}(s),\widetilde{\boldsymbol{w}}(s)\rangle_s,\ \ s\in[0,T],
\end{eqnarray}
\end{lem}
where $G$ was defined as in Section 2.
\vskip 0.4cm
Combining the above lemma  with (\ref{eq w}), we have
\begin{lem}\label{lem 6}
For any $i,j\in\mathbb{N}$,
\begin{eqnarray}\label{eq lem 6 1}
0&=&\frac{d\langle\widetilde{\boldsymbol{w}}_i(s),\widetilde{\boldsymbol{w}}_j(s)\rangle_s}{ds}\nonumber\\
&=&
\langle\widetilde{\boldsymbol{w}}'_i(s)+G\widetilde{\boldsymbol{w}}_i(s),\widetilde{\boldsymbol{w}}_j(s)\rangle_s
+
\langle\widetilde{\boldsymbol{w}}'_j(s)+G\widetilde{\boldsymbol{w}}_j(s),\widetilde{\boldsymbol{w}}_i(s)\rangle_s,
 \ s\in[0,T].
\end{eqnarray}
\end{lem}
\begin{proof}
(\ref{eq w}) obviously implies the first equality in (\ref{eq lem 6 1}). Now replace $\widetilde{\boldsymbol{w}}$ in Lemma \ref{lem 5} by $\widetilde{\boldsymbol{w}}_i+\widetilde{\boldsymbol{w}}_j$
and $\widetilde{\boldsymbol{w}}_i-\widetilde{\boldsymbol{w}}_j$, respectively, to get
\begin{eqnarray}\label{eq lem 5 2}
\frac{d|\widetilde{\boldsymbol{w}}_i(s)+\widetilde{\boldsymbol{w}}_j(s)|_s^2}{ds}
=
2\langle\widetilde{\boldsymbol{w}}_i'(s)+\widetilde{\boldsymbol{w}}_j'(s)
   +
    G(\widetilde{\boldsymbol{w}}_i(s)+\widetilde{\boldsymbol{w}}_j(s)),\widetilde{\boldsymbol{w}}_i(s)+\widetilde{\boldsymbol{w}}_j(s)\rangle_s,
\end{eqnarray}
and
\begin{eqnarray}\label{eq lem 5 3}
\frac{d|\widetilde{\boldsymbol{w}}_i(s)-\widetilde{\boldsymbol{w}}_j(s)|_s^2}{ds}
=
2\langle\widetilde{\boldsymbol{w}}_i'(s)-\widetilde{\boldsymbol{w}}_j'(s)
   +
    G(\widetilde{\boldsymbol{w}}_i(s)-\widetilde{\boldsymbol{w}}_j(s)),\widetilde{\boldsymbol{w}}_i(s)-\widetilde{\boldsymbol{w}}_j(s)\rangle_s.
\end{eqnarray}
Using the fact that
$$
4\langle\widetilde{\boldsymbol{w}}_i(s),\widetilde{\boldsymbol{w}}_j(s)\rangle_s
=
|\widetilde{\boldsymbol{w}}_i(s)+\widetilde{\boldsymbol{w}}_j(s)|_s^2
-
|\widetilde{\boldsymbol{w}}_i(s)-\widetilde{\boldsymbol{w}}_j(s)|_s^2
$$
we have
\begin{eqnarray*}
&&\frac{d4\langle\widetilde{\boldsymbol{w}}_i(s),\widetilde{\boldsymbol{w}}_j(s)\rangle_s}{ds}\nonumber\\
&=&
2\langle\widetilde{\boldsymbol{w}}_i'(s)+\widetilde{\boldsymbol{w}}_j'(s)
   +
    G\widetilde{\boldsymbol{w}}_i(s)+G\widetilde{\boldsymbol{w}}_j(s),\widetilde{\boldsymbol{w}}_i(s)+\widetilde{\boldsymbol{w}}_j(s)\rangle_s\nonumber\\
&&-
2\langle\widetilde{\boldsymbol{w}}_i'(s)-\widetilde{\boldsymbol{w}}_j'(s)
   +
    G\widetilde{\boldsymbol{w}}_i(s)-G\widetilde{\boldsymbol{w}}_j(s),\widetilde{\boldsymbol{w}}_i(s)-\widetilde{\boldsymbol{w}}_j(s)\rangle_s\nonumber\\
&=&
4\langle\widetilde{\boldsymbol{w}}'_i(s)+G\widetilde{\boldsymbol{w}}_i(s),\widetilde{\boldsymbol{w}}_j(s)\rangle_s
+
4\langle\widetilde{\boldsymbol{w}}'_j(s)+G\widetilde{\boldsymbol{w}}_j(s),\widetilde{\boldsymbol{w}}_i(s)\rangle_s.
\end{eqnarray*}

 The proof of Lemma \ref{lem 6} is complete.
\end{proof}


\vskip 0.4cm

For each $m\in\mathbb{N}$, we define an approximate process $\widetilde{\boldsymbol{u}}_m$ as follows:
\begin{align}\label{eq um}
\widetilde{\boldsymbol{u}}_m(s)=\sum^{m}_{j=1}  {g}_{jm}(s) \widetilde{\boldsymbol{w}}_j(s),\ \ \ s\in[0,T],
\end{align}
where ${g}_{jm}, 1\leq j\leq m$ are real-valued  stochastic processes on $[0,T]$ determined by the equations:
\begin{align}
\left<\mathrm{d}\widetilde{\boldsymbol{u}}_m(s),\widetilde{\boldsymbol{w}}_j(s)\right>_s
&=\left<F\widetilde{\boldsymbol{u}}_m(s),\widetilde{\boldsymbol{w}}_j(s)\right>_s\mathrm{d}s
-\left<G\widetilde{\boldsymbol{u}}_m(s),\widetilde{\boldsymbol{w}}_j(s)\right>_s\mathrm{d}s-\left<\mathcal{N}(\widetilde{\boldsymbol{u}}_m(s)),  \widetilde{\boldsymbol{w}}_j(s)\right>_s\mathrm{d}s\nonumber\\ &\ \ +\left<\widetilde{\boldsymbol{f}}(s),\widetilde{\boldsymbol{w}}_j(s)\right>_s\mathrm{d}s
+\left<\widetilde{\boldsymbol{\sigma}}(s),\widetilde{\boldsymbol{w}}_j(s)\right>_s\mathrm{d}W(s),
\ j=1,2,\cdots,m,\label{equ1}
\end{align}
with initial data
$$
\widetilde{\boldsymbol{u}}_m(0)=\sum^{m}_{j=1}  {g}_{jm}(0) \widetilde{\boldsymbol{w}}_j(0),
$$
where
$$
{g}_{jm}(0)=\left<\widetilde{\boldsymbol{u}}_0,\widetilde{\boldsymbol{w}}_{j}(0)\right>_0, \ j=1,2,\cdots,m.
$$

Since $F$ and $G$ are linear operators, for $j=1,2,\cdots, m,$
\begin{equation*}
\left<F\widetilde{\boldsymbol{u}}_m(s),\widetilde{\boldsymbol{w}}_j(s)\right>_s=\sum^{m}_{k=1}{g}_{km}(s)\left<F\widetilde{\boldsymbol{w}}_k(s),\widetilde{\boldsymbol{w}}_j(s)\right>_s,
\end{equation*}
\begin{equation*}
\left<G\widetilde{\boldsymbol{u}}_m(s),\widetilde{\boldsymbol{w}}_j(s)\right>_s=\sum^{m}_{k=1}{g}_{km}(s)\left<G\widetilde{\boldsymbol{w}}_k(s),\widetilde{\boldsymbol{w}}_j(s)\right>_s.
\end{equation*}
And for the bilinear operator $\mathcal{N}$, we have, for $j=1,2,\cdots, m,$
\begin{align*}
\left<\mathcal{N}(\widetilde{\boldsymbol{u}}_m(s)),  \widetilde{\boldsymbol{w}}_j(s)\right>_s&=\left<\mathcal{N}(\widetilde{\boldsymbol{u}}_m(s),\widetilde{\boldsymbol{u}}_m(s)),  \widetilde{\boldsymbol{w}}_j(s)\right>_s\\
&=\left<\sum_{i=1}^{2} \left(\sum^{m}_{k=1}  {g}_{km}(s) \widetilde{\boldsymbol{w}}_k(s)\right)^i\nabla_i \left(\sum^{m}_{l=1}  {g}_{lm}(s) \widetilde{\boldsymbol{w}}_l(s)\right), \widetilde{\boldsymbol{w}}_j(s)\right>_s\\
&=\sum^{m}_{k=1}\sum^{m}_{l=1}  {g}_{km}(s)   {g}_{lm}(s) \left<\sum_{i=1}^{2} \left( \widetilde{\boldsymbol{w}}_k(s)\right)^i\nabla_i \left( \widetilde{\boldsymbol{w}}_l(s)\right), \widetilde{\boldsymbol{w}}_j(s)\right>_s\\
&=\sum^{m}_{k=1}\sum^{m}_{l=1} {g}_{km}(s){g}_{lm}(s)\left< \mathcal{N} (\widetilde{\boldsymbol{w}}_k(s),\widetilde{\boldsymbol{w}}_l(s)),\widetilde{\boldsymbol{w}}_j(s)\right>_s.
\end{align*}
Thus, equivalently ${g}_{jm}, 1\leq j\leq m$ solve the following  system of  stochastic differential equations on $[0,T]$:
\begin{align}\label{gkm}
&\mathrm{d}{g}_{jm}(s)+\sum_{k=1}^{m}{g}_{km}(s) \left<\boldsymbol{\widetilde{w}_{k}}^{\prime}(s),\boldsymbol{\widetilde{w}}_{j}(s)\right>_s \mathrm{d} s-\sum^{m}_{k=1}{g}_{km}(s)\left<F\widetilde{\boldsymbol{w}}_k(s),\widetilde{\boldsymbol{w}}_j(s)\right>_s\mathrm{d}s \nonumber \\
& + \sum^{m}_{k=1}{g}_{km}(s)\left<G\widetilde{\boldsymbol{w}}_k(s),\widetilde{\boldsymbol{w}}_j(s)\right>_s\mathrm{d}s+ \sum^{m}_{k=1}\sum^{m}_{l=1} {g}_{km}(s){g}_{lm}(s)\left< \mathcal{N} (\widetilde{\boldsymbol{w}}_k(s),\widetilde{\boldsymbol{w}}_l(s)),\widetilde{\boldsymbol{w}}_j(s)\right>_s\mathrm{d}s  \nonumber  \\
=&\left<\widetilde{\boldsymbol{f}}(s),\widetilde{\boldsymbol{w}}_j(s)\right>_s\mathrm{d}s  +\left<\widetilde{\boldsymbol{\sigma}}(s),\widetilde{\boldsymbol{w}}_j(s)\right>_s\mathrm{d}W(s),\ \ j=1,2,\cdots,m,
\end{align}
with the initial condition
\begin{equation*}
{g}_{jm}(0)=\left<\widetilde{\boldsymbol{u}}_0,\widetilde{\boldsymbol{w}}_{j}(0)\right>_0,\ \ j=1,2,\cdots,m.
\end{equation*}
\begin{lem}\label{Lem 4.3}
There exists a unique  solution  ${g}_{jm}, 1\leq j\leq m$ to equation (\ref{gkm}).
\end{lem}
\begin{proof}
For $ k,l=1,2,\cdots,m$, set
\begin{equation*}
a_{jk}(s)=\left<\boldsymbol{\widetilde{w}}_{k}^{\prime}(s),\boldsymbol{\widetilde{w}}_{j}(s)\right>_s-\left<F\widetilde{\boldsymbol{w}}_k(s),\widetilde{\boldsymbol{w}}_j(s)\right>_s+\left<G\widetilde{\boldsymbol{w}}_k(s),\widetilde{\boldsymbol{w}}_j(s)\right>_s,
\end{equation*}
\begin{equation*}
a_{jkl}(s)=\left< \mathcal{N} (\widetilde{\boldsymbol{w}}_k(s),\widetilde{\boldsymbol{w}}_l(s)),\widetilde{\boldsymbol{w}}_j(s)\right>_s.
\end{equation*}
The equation (\ref{gkm}) can be expressed as
\begin{align}
&\mathrm{d}{g}_{jm}(s)+\sum_{k=1}^{m}a_{jk}(s) {g}_{km}(s)  \mathrm{d}s+ \sum^{m}_{k=1}\sum^{m}_{l=1} a_{jkl}(s){g}_{km}(s){g}_{lm}(s) \mathrm{d}s  \nonumber  \\
=&\left<\widetilde{\boldsymbol{f}}(s),\widetilde{\boldsymbol{w}}_j(s)\right>_s \mathrm{d}s  +\left<\widetilde{\boldsymbol{\sigma}}(s),\widetilde{\boldsymbol{w}}_j(s)\right>_s \mathrm{d}W(s),\ \ j=1,2,\cdots,m.\label{SDE}
\end{align}
For $\boldsymbol{y}=(y^1,y^2,\cdots,y^m)\in \mathbb{R}^m$, since the coefficients $a_{jk}(s) y^k$, $ a_{jkl}(s)y^ky^l $ are locally Lipschitz, equation (\ref{SDE}) admits a unique local solution. The global existence and uniqueness of the  solution ${g}_{jm}, 1\leq j\leq m$ follow from the \emph{a priori}  estimates proved later in Lemma \ref{ener} below for  $\widetilde{\boldsymbol{u}}_m$.

The proof of Lemma \ref{Lem 4.3} is complete.
\end{proof}
\vskip 0.3cm
To obtain the estimate for   $\widetilde{\boldsymbol{u}}_m$,  we need the  following  chain rule.

\begin{prop}\label{prop1}
We have the following chain rule for  $\widetilde{\boldsymbol{u}}_m$:
\begin{align}
\mathrm{d}|\widetilde{\boldsymbol{u}}_m(s)|_s^2=&2\left<F\widetilde{\boldsymbol{u}}_m(s), \widetilde{\boldsymbol{u}}_m(s)\right>_s\mathrm{d}s
-2\left<\mathcal{N}(\widetilde{\boldsymbol{u}}_m(s)),  \widetilde{\boldsymbol{u}}_m(s)\right>_s\mathrm{d}s\nonumber\\&+2\left<\widetilde{\boldsymbol{f}}(s),\widetilde{\boldsymbol{u}}_m(s)\right>_s\mathrm{d}s+2\left<\widetilde{\boldsymbol{\sigma}}(s),\widetilde{\boldsymbol{u}}_m(s)\right>_s\mathrm{d}W(s)+|\widetilde{\boldsymbol{\sigma}}_m(s)|_s^2 \mathrm{d}s.\label{cr}
\end{align}

\end{prop}
\begin{proof} Recall (\ref{gkm}).
By the It\^o formula,
\begin{align}
\sum_{j=1}^{m}\mathrm{d}{g}_{jm}^2(s)
=&2\sum_{j=1}^{m}{g}_{jm}(s)\mathrm{d}{g}_{jm}(s)+\sum_{j=1}^{m}|\left<\widetilde{\boldsymbol{\sigma}}(s),\widetilde{\boldsymbol{w}}_j(s)\right>_s|^2 \mathrm{d}s\nonumber\\
=&-2\sum_{j=1}^{m}\sum_{k=1}^{m}{g}_{km}(s){g}_{jm}(s) \left<\boldsymbol{\widetilde{w}}_{k}^{\prime}(s),\boldsymbol{\widetilde{w}}_{j}(s)\right>_s \mathrm{d}s\nonumber\\
&+2\sum_{j=1}^{m}\sum^{m}_{k=1}{g}_{km}(s){g}_{jm}(s)\left<F\widetilde{\boldsymbol{w}}_k(s),\widetilde{\boldsymbol{w}}_j(s)\right>_s\mathrm{d}s
\nonumber\\
&-2\sum_{j=1}^{m}\sum^{m}_{k=1}{g}_{km}(s){g}_{jm}(s)\left<G\widetilde{\boldsymbol{w}}_k(s),\widetilde{\boldsymbol{w}}_j(s)\right>_s\mathrm{d}s\nonumber
\\&-2\sum_{j=1}^{m}\sum^{m}_{k=1}\sum^{m}_{l=1}{g}_{jm}(s) {g}_{km}(s){g}_{lm}(s)\left< \mathcal{N} (\widetilde{\boldsymbol{w}}_k(s),\widetilde{\boldsymbol{w}}_l(s)),\widetilde{\boldsymbol{w}}_j(s)\right>_s \mathrm{d}s \nonumber
\\&+2\sum_{j=1}^{m}\left<\widetilde{\boldsymbol{f}}(s),\widetilde{\boldsymbol{w}}_j(s)\right>_s{g}_{jm}(s)\mathrm{d}s+2\sum_{j=1}^{m}\left<\widetilde{\boldsymbol{\sigma}}(s),\widetilde{\boldsymbol{w}}_j(s)\right>_s{g}_{jm}(s)\mathrm{d}W(s)\nonumber\\
&+\sum_{j=1}^{m}|\left<\widetilde{\boldsymbol{\sigma}}(s),\widetilde{\boldsymbol{w}}_j(s)\right>_s|^2 \mathrm{d}s\nonumber\\
=&-2\sum_{j=1}^{m}\sum_{k=1}^{m}{g}_{km}(s){g}_{jm}(s) \left<\boldsymbol{\widetilde{w}}_{k}^{\prime}(s),\boldsymbol{\widetilde{w}}_{j}(s)\right>_s \mathrm{d}s\nonumber\\
&-2\sum_{j=1}^{m}\sum^{m}_{k=1}{g}_{km}(s){g}_{jm}(s)\left<G\widetilde{\boldsymbol{w}}_k(s),\widetilde{\boldsymbol{w}}_j(s)\right>_s\mathrm{d}s\nonumber\\
&+2\left<\sum^{m}_{k=1}{g}_{km}(s)F\widetilde{\boldsymbol{w}}_k(s),\sum_{j=1}^{m}{g}_{jm}(s) \widetilde{\boldsymbol{w}}_j(s)\right>_s\mathrm{d}s\nonumber\\
&-2\left<\sum^{m}_{k=1}\sum^{m}_{l=1} {g}_{km}(s){g}_{lm} \mathcal{N} (\widetilde{\boldsymbol{w}}_k(s),\widetilde{\boldsymbol{w}}_l(s)),\sum_{j=1}^{m} {g}_{jm}(s) \widetilde{\boldsymbol{w}}_j(s)\right>_s \mathrm{d}s  \nonumber\\&+2\left<\widetilde{\boldsymbol{f}}(s),\sum^{m}_{j=1}  {g}_{jm}(s) \widetilde{\boldsymbol{w}}_k(s)\right>_s\mathrm{d}s+2\left<\widetilde{\boldsymbol{\sigma}}(s),\sum^{m}_{j=1}  {g}_{jm}(s) \widetilde{\boldsymbol{w}}_k(s)\right>_s\mathrm{d}W(s)\nonumber\\
&+\sum_{j=1}^{m}|\left<\widetilde{\boldsymbol{\sigma}}(s),\widetilde{\boldsymbol{w}}_j(s)\right>_s|^2 \mathrm{d}s\nonumber\\
=&-2\sum_{j=1}^{m}\sum_{k=1}^{m}{g}_{km}(s){g}_{jm}(s) \left<\boldsymbol{\widetilde{w}}_{k}^{\prime}(s),\boldsymbol{\widetilde{w}}_{j}(s)\right>_s \mathrm{d}s\nonumber\\
&-2\sum_{j=1}^{m}\sum^{m}_{k=1}{g}_{km}(s){g}_{jm}(s)\left<G\widetilde{\boldsymbol{w}}_k(s),\widetilde{\boldsymbol{w}}_j(s)\right>_s\mathrm{d}s\nonumber\\
&+2\left<F\widetilde{\boldsymbol{u}}_m(s), \widetilde{\boldsymbol{u}}_m(s)\right>_s\mathrm{d}s
-2\left<\mathcal{N}(\widetilde{\boldsymbol{u}}_m(s)),  \widetilde{\boldsymbol{u}}_m(s)\right>_s\mathrm{d}s+2\left<\widetilde{\boldsymbol{f}}(s),\widetilde{\boldsymbol{u}}_m(s)\right>_s\mathrm{d}s\nonumber\\&+2\left<\widetilde{\boldsymbol{\sigma}}(s),\widetilde{\boldsymbol{u}}_m(s)\right>_s\mathrm{d}W(s)+|\widetilde{\boldsymbol{\sigma}}_m(s)|_s^2 \mathrm{d}s.\label{ito-formula}
\end{align}
In the last equality of (\ref{ito-formula}), we have used (\ref{eq um}), and the facts that $F$ is linear operator and $\mathcal{N}$ is a bilinear operator.

Observing that, (\ref{eq um}), (\ref{eq w}) and Lemma \ref{lem 6} imply that
\begin{eqnarray*}
|\widetilde{\boldsymbol{u}}_m(s)|_s^2
=
\sum_{j=1}^{m}{g}_{jm}^2(s),
\end{eqnarray*}
and
\begin{eqnarray*}
2\sum_{j=1}^{m}\sum_{k=1}^{m}{g}_{km}(s){g}_{jm}(s) \left<\boldsymbol{\widetilde{w}}_{k}^{\prime}(s),\boldsymbol{\widetilde{w}}_{j}(s)\right>_s \mathrm{d}s
+
2\sum_{j=1}^{m}\sum^{m}_{k=1}{g}_{km}(s){g}_{jm}(s)\left<G\widetilde{\boldsymbol{w}}_k(s),\widetilde{\boldsymbol{w}}_j(s)\right>_s\mathrm{d}s
=
0,
\end{eqnarray*}
it follows from (\ref{ito-formula}) that
\begin{align*}
\mathrm{d}|\widetilde{\boldsymbol{u}}_m(s)|_s^2=&2\left<F\widetilde{\boldsymbol{u}}_m(s), \widetilde{\boldsymbol{u}}_m(s)\right>_s\mathrm{d}s
-2\left<\mathcal{N}(\widetilde{\boldsymbol{u}}_m(s)),  \widetilde{\boldsymbol{u}}_m(s)\right>_s\mathrm{d}s\\&+2\left<\widetilde{\boldsymbol{f}}(s),\widetilde{\boldsymbol{u}}_m(s)\right>_s\mathrm{d}s+2\left<\widetilde{\boldsymbol{\sigma}}_m(s),\widetilde{\boldsymbol{u}}_m(s)\right>_s\mathrm{d}W(s)+|\widetilde{\boldsymbol{\sigma}}_m(s)|_s^2 \mathrm{d}s.
\end{align*}

The proof of Proposition \ref{prop1} is complete.
\end{proof}
\vskip 0.4cm
Recall
$$\boldsymbol{u}_m(x,t)=\widetilde{\boldsymbol{u}}_{m}(L(x,t))M(x,t)^{-1}.$$
Next result provides a uniform bound for the family $\{\boldsymbol{u}_m, m\geq 1\}$.
\begin{lem}
There exists a constant $M$, independent of $m$,  such that
\begin{equation}
 \mathbb{E} \sup_{0\leq t\leq T}\|\boldsymbol{u}_m(t)\|_{t}^2+\mathbb{E}\int_{0}^{T}\|\nabla \boldsymbol{  u}_m(t)\|_t^2\mathrm{d} t \leq M.\label{estimate1}
\end{equation} \label{ener}
\end{lem}
\begin{proof}
By a change of variable, we see that
\begin{equation*}
\left<F\widetilde{\boldsymbol{u}}_m(t), \widetilde{\boldsymbol{u}}_m(t) \right>_t \mathrm{d}t= \left(\Delta \boldsymbol{u}_m(t), \boldsymbol{u}_m(t)\right)_t  \mathrm{d} t=-\|\nabla \boldsymbol{ u}_m(t)\|_t^2 \mathrm{d} t,
\end{equation*}
\begin{equation*}
\left<\mathcal{N}(\widetilde{\boldsymbol{u}}_m(t)),  \widetilde{\boldsymbol{u}}_m(t)\right>_t\mathrm{d}t = \left(\left(\boldsymbol{u}_m(t)\cdot\nabla\right)\boldsymbol{u}_m(t),\boldsymbol{u}_m(t)\right)_t\mathrm{d}t=0,
\end{equation*}
and (\ref{eq 1}) implies
$$
|\widetilde{\boldsymbol{u}}_m(t)|_t^2
=
\|\boldsymbol{u}_m(t)\|_{t}^2
\text{ and }
|\widetilde{\boldsymbol{\sigma}}_m(t)|_t^2
=
\|\boldsymbol{\sigma}_m(t)\|_{t}^2,
$$
Equation (\ref{cr}) becomes
\begin{equation}\label{energy}
\mathrm{d}\|\boldsymbol{u}_m(t)\|_{t}^2+ 2\| \nabla \boldsymbol{ u}_m(t)\|_t^2 \mathrm{d} t=2\left(\boldsymbol{f}(t),\boldsymbol{u}_m(t)\right)_t \mathrm{d} t+2\left(\boldsymbol{\sigma}_m(t),\boldsymbol{u}_m(t)\right)_t \mathrm{d} W(t)+\|\boldsymbol{\sigma}_m(t)\|_t^2 \mathrm{d} t.
\end{equation}

Hence, we have
\begin{equation}
\mathrm{d} \|\boldsymbol{u}_m(t)\|_t^2+ 2\|\nabla  \boldsymbol{ u}_m(t)\|_t^2 \mathrm{d} t\leq  \|\boldsymbol{f}(t)\|_{_t}^{*^2} \mathrm{d} t + \|\nabla \boldsymbol{u}_m(t) \|_{t}^2 \mathrm{d} t+2\left(\boldsymbol{\sigma}_m(t),\boldsymbol{u}_m(t)\right)_t \mathrm{d} W(t)+\|\boldsymbol{\sigma}_m(t)\|_t^2 \mathrm{d} t. \label{estimate}
\end{equation}
By the Burkholder-Davis-Gundy inequality, there exists some constant $C_4$ such that
\begin{equation*}
\mathbb{E}\left\{ \sup_{0\leq t\leq T} \left|\int_{0}^{t}\left(\boldsymbol{\sigma}(s),\boldsymbol{u}_m(s)\right)_s \mathrm{d} W(s)\right| \right\}\leq C_4 \mathbb{E}\left\{  \left(\int_{0}^{T}\left(\boldsymbol{\sigma}(s),\boldsymbol{u}_m(s)\right)_s^2 \mathrm{d} s \right)^{1/2} \right\}.
\end{equation*}
It follows from  (\ref{estimate}) that
\begin{align*}
 &\mathbb{E} \sup_{0\leq t\leq T}\|\boldsymbol{u}_m(t)\|_{t}^2+2\mathbb{E}\int_{0}^{T}\|\nabla  \boldsymbol{ u}_m(t)\|_t^2\mathrm{d} t\nonumber\\
\leq &\mathbb{E}\|\boldsymbol{u}_m(0)\|_{0}^2+ \mathbb{E}\int_{0}^{T} \|\boldsymbol{f}(t)\|_{t}^{*^2}\mathrm{d}t
+\mathbb{E}\int_{0}^{T} \|\nabla  \boldsymbol{u_m}(t)\|_{t}^{2}\mathrm{d}t+\mathbb{E}\int_{0}^{T} \|\boldsymbol{\sigma_m}(t)\|_{t}^{2}\mathrm{d}t\\&+C_4 \mathbb{E}\left\{  \left(\int_{0}^{T}\left(\boldsymbol{\sigma}_m(s),\boldsymbol{u}_m(s)\right)_s^2 \mathrm{d} s \right)^{1/2} \right\},\\
\leq &\mathbb{E}\|\boldsymbol{u}_m(0)\|_{0}^2+ \mathbb{E}\int_{0}^{T} \|\boldsymbol{f}(t)\|_{t}^{*^2}\mathrm{d}t
+\mathbb{E}\int_{0}^{T} \|\nabla \boldsymbol{ u_m}(t)\|_{t}^{2}\mathrm{d}t+\mathbb{E}\int_{0}^{T} \|\boldsymbol{\sigma_m}(t)\|_{t}^{2}\mathrm{d}t\\&+C_4 \mathbb{E}\left\{ \sup_{0\leq t\leq T}\|\boldsymbol{u}_m(t)\|_t \left(\int_{0}^{T}\|\boldsymbol{\sigma}_m(s)\|_s^2 \mathrm{d} s \right)^{1/2} \right\}\\
\leq &\mathbb{E}\|\boldsymbol{u}_m(0)\|_{0}^2+ \mathbb{E}\int_{0}^{T} \|\boldsymbol{f}(t)\|_{t}^{*^2}\mathrm{d}t
+\mathbb{E}\int_{0}^{T} \|\nabla \boldsymbol{ u_m}(t)\|_{t}^{2}\mathrm{d}t+\mathbb{E}\int_{0}^{T} \|\boldsymbol{\sigma_m}(t)\|_{t}^{2}\mathrm{d}t\\&
+\frac{1}{2}\mathbb{E} \sup_{0\leq t\leq T}\|\boldsymbol{u}_m(t)\|_{t}^2+C_5 \mathbb{E}\int_{0}^{T}\|\boldsymbol{\sigma}_m(s)\|_s^2 \mathrm{d} s.
\end{align*}
Re-arranging the terms, we get
\begin{equation}
 \mathbb{E} \sup_{0\leq t\leq T}\|\boldsymbol{u}_m(t)\|_{t}^2
 +
 2\mathbb{E}\int_{0}^{T}\|\nabla  \boldsymbol{ u}_m(t)\|_t^2\mathrm{d} t
  \leq C_6.
\end{equation}

The proof of Lemma \ref{ener} is complete.
\end{proof}

\vskip 0.2cm
To obtain a martingale solution to the stochastic Navier-Stokes equation (\ref{add}),  we will prove that the family of laws  $\{\mathscr{L}(\boldsymbol{u}_m), m\geq 1\}$ is tight in $\mathbb{L}^2([0,T];H_t)$.
Let
\begin{eqnarray}\label{eq w00}
\boldsymbol{w}_j (x,t):=\widetilde{\boldsymbol{w}}_j(L(x,t))M(x,t)^{-1}, j\geq 1.
\end{eqnarray}
By (\ref{eq 1}) and (\ref{eq w}), $\{\boldsymbol{w}_j\}_{j\in\mathbb{N}}$
is an orthogonal basis with respect to the inner product (\ref{eq 3.11}). The following lemma will be used.
\begin{lem}\label{lem 4.5}(Lemma 2.5 in  \cite{TY}) For each $\varepsilon>0$, there exists a positive integer $N=N_{\varepsilon}$ independent of $t\in[0,T]$ such that for any $\boldsymbol{v}\in V_t$,  we have
\begin{equation}
\left\|\boldsymbol{v}\right\|_{t}^2\leq \sum\limits_{j=1}^{N}\left(\boldsymbol{v},\boldsymbol{w_j}(t)\right)_t^2+\varepsilon\|\nabla \boldsymbol{v}\|_{t}^2.
\end{equation}
\end{lem}
\vskip 0.2cm

Next we are going to introduce a family of precompact subsets of $\mathbb{L}^2([0,T];H_t)$. For $i\geq 1$, $J_i$ denotes a family of  equicontinuous functions on $[0,T]$. Recall that $J_i$ is said to be  \emph{equicontinuous} on $[0,T]$ if for $\forall \varepsilon>0$, $ \exists  \delta>0$ such that $\forall  s, t\in [0,T]$ with $|s-t|<\delta$, $|h(t)-h(s)|< \varepsilon$,  $\forall h\in J_i$.
For $N>0$, set $K_{N,J}=\bigcap_{i=1}^{\infty}K_{N,J_i}$, where
\begin{eqnarray}
K_{N,J_i}=&\Big\{\boldsymbol{g}\in\mathbb{L}^{\infty}\left([0,T];H_t\right)\cap\mathbb{L}^{2}\left([0,T];V_t\right):\sup\limits_{0\leq t\leq T}\left\|\boldsymbol{g}(t)\right\|_{t}\leq N, \int_{0}^{T} \left\|\nabla\boldsymbol{g}(t)\right\|_{t}^2\mathrm{d}t\leq N,\nonumber\\ &\boldsymbol{g}_i=\left\{\boldsymbol{g}_i(t)=(\boldsymbol{g}\left(t\right),\boldsymbol{w}_i\left(t\right))_{t}, t\in[0,T]\right\}\in J_i \Big\}.
\end{eqnarray}

\begin{prop}
$K_{N,J}$ is precompact in $\mathbb{L}^2([0,T];H_t)$.
\end{prop}
\begin{proof}
For each sequence  $\left\{\boldsymbol{z}_m\right\}_{m\geq1}\in K_{N,J}$, set $\boldsymbol{\rho}_{m,j}(t)=\left(\boldsymbol{z}_m(t),\boldsymbol{w}_j(t)\right)_t$ for all $j\geq 1$. By the definition of $K_{N,J}$, $\{\boldsymbol{\rho}_{m,j}(t),t\in[0,T]\}_{m\geq 1}$ is equicontinuous. Noting that (\ref{eq w}) and (\ref{eq w00}) imply that $\boldsymbol{w}_j$ is smooth in $(x,t)$, there exists a constant $M_j$ such that
\begin{equation}\label{eq Mj}
\left|\boldsymbol{w}_j(x,t)\right|\leq M_j,\left|\nabla \boldsymbol{w}_j(x,t)\right|\leq M_j,\left|\frac{\partial}{\partial t} \boldsymbol{w}_j(x,t)\right|\leq M_j,\quad \quad \forall x\in \mathcal{D}(t),t\in [0,T],
\end{equation}
which implies
\begin{align*}
\left|\boldsymbol{\rho}_{m,j}(t)\right|\leq  \left\|\boldsymbol{z}_m(t)\right\|_{t}M_j|\mathcal{D}(t)|^{1/2}
\leq   M_{j,N},
\end{align*}
where $|\mathcal{D}(t)|$ denotes the volume of $\mathcal{D}(t)$, i.e. $|\mathcal{D}(t)|=\int_{\mathcal{D}(t)}1dx$.
The last inequality is obtained from the facts that  $\sup\limits_{m\geq 1}\sup\limits_{0\leq t\leq T}\left\|\boldsymbol{z}_m(t)\right\|_{t}\leq N$ and the condition (A1).

Applying the Arzel\`a-Ascoli theorem and a diagonalization argument, we can choose a subsequence  $\{m_k, k\geq 1\}$ such that   $\left\{\{\boldsymbol{\rho}_{m_k,j}(t), t\in[0,T]\}, k\geq 1\right\}$ is a Cauchy sequence in $C([0,T];\mathbb{R})$ for each fixed $j$.\\
  \indent From Lemma \ref{lem 4.5}, for any $\varepsilon>0$, we have
\begin{equation}
\int_{0}^{T}\|\boldsymbol{z}_{m_k}(t)-\boldsymbol{z}_{m_l}(t)\|_{t}^2 \mathrm{d}t \leq \sum_{j=1}^{N_{\varepsilon}} \int_{0}^{T} \left|\rho_{m_k,j}(t)-\rho_{m_l,j}(t)\right|^2 \mathrm{d}t+2\varepsilon\sup\limits_{m\geq 1} \int_{0}^{T}\| \nabla\boldsymbol{z}_m(t)\|_{t}^2\mathrm{d}t.
\end{equation}
Let $k,l\to \infty$ to get
\begin{equation*}
\limsup\limits_{k,l\to \infty}\int_{0}^{T} \|\boldsymbol{z}_{m_k}(t)-\boldsymbol{z}_{m_l}(t)\|_{t}^2 \mathrm{d}t\leq 2\varepsilon N.
\end{equation*}
Since  $\varepsilon$ is arbitrary, we see that $\{\boldsymbol{z}_{m_k}\}_{k\geq 1}$ is a Cauchy sequence in $\mathbb{L}^2([0,T];H_t)$, completing the proof.
\end{proof}
\vskip 0.5cm
Recall that $\{\boldsymbol{u}_m\}_{m=1}^{\infty}$ satisfy the equation: for $1\leq j\leq m$ and $t\in[0,T]$,
\begin{align}&(\boldsymbol{u}_m\left(t\right),\boldsymbol{w}_j\left(t\right))_{t}-(\boldsymbol{u}_m\left(0\right),\boldsymbol{w}_j\left(0\right))_{0}
+\int_{0}^{t} (\nabla \boldsymbol{u}_m\left(s\right),\nabla \boldsymbol{w}_j\left(s\right))_{s} \mathrm{d}s-\int_{0}^{t} \ _{V^*_s}\left<\boldsymbol{f}(s),\boldsymbol{w}_j(s)\right>_{V_s}  \mathrm{d} s \nonumber\\
&+\int_{0}^{t} b_s\left(\boldsymbol{u}_m\left(s\right),\boldsymbol{u}_m\left(s\right),\boldsymbol{w}_j\left(s\right)\right) \mathrm{d}s
-\int_{0}^{t}(\boldsymbol{u}_m(s),\boldsymbol{w}_j^{\prime}(s))_{s} \mathrm{d}s=\int_{0}^{t} (\boldsymbol{\sigma}(s),\boldsymbol{w}_j(s))_{s}\mathrm{d}W(s).\label{component}
\end{align}

\begin{prop}
\{$\mathscr{L}(\boldsymbol{u}_m)$,$m\geq 1$\} is tight in $\mathbb{L}^2([0,T];H_t)$.
\end{prop}
\begin{proof}
Let $K_{N,J}$ be defined as in Proposition 4.2. We just need to prove: for $\forall \varepsilon>0$, there exist $J_i, i\in\mathbb{N}$ and $N$ such that  $\mathbb{P} (\boldsymbol{u}_m\in K_{N,J})\geq 1-\varepsilon$.

Set
\begin{equation}
X_j(t)=\int_{0}^{t} (\boldsymbol{\sigma}(s),\boldsymbol{w}_j(s))_{s} \mathrm{d}W(s).
\end{equation}
Recall $p>2$ in Theorem \ref{thm1}. By the Burkholder-Davis-Gundy inequality, the H\"older inequality and (\ref{eq Mj}), for every $n\geq 1$ and $0\leq s\leq t\leq T$,
\begin{align*}
\mathbb{E}|X_j(t)-X_j(s)|^{2n}&\leq \mathbb{E}|\int_{s}^{t} (\boldsymbol{\sigma}(l),\boldsymbol{w}_j(l))_{l} \mathrm{d}W(l)|^{2n}\\
&\leq  C_{n} \left(\int_{s}^{t} (\boldsymbol{\sigma}(l),\boldsymbol{w}_j(l))^2_{l} \mathrm{d}l\right)^{n}\\
& \leq C_{j,n} |t-s|^{\frac{n(p-2)}{p}}   (\int_0^T\|\boldsymbol{\sigma}(l) \|_l^pdl)^{\frac{2n}{p}}\\
&\leq C_{j,n}^{\prime} |t-s|^{\frac{n(p-2)}{p}}.
\end{align*}
The condition that $\boldsymbol{\sigma}\in \mathbb{L}^p([0,T];H_{t})$ has been used.

  Letting $n=\frac{4p}{p-2}$ and applying the Garsia lemma (see Corollary 1.2 in \cite{Walsh}), there exists a random variable $Y_j$ such that  with probability one, for all $0\leq s\leq t\leq T$,
\begin{align*}
|X_j(t,\omega)-X_j(s,\omega)|
&\leq Y_j(\omega) |t-s|^{\frac{p-2}{4p}},
\end{align*}
where
\begin{equation*}
\mathbb{E}(Y_j^{\frac{8p}{p-2}})< \infty.
\end{equation*}
It follows from (\ref{eq b}), (\ref{eq Mj}) and (\ref{component}) that, for $0\leq s\leq t\leq T$,
\begin{align}\label{component1}
&\ |(\boldsymbol{u}_m(t),\boldsymbol{w}_j(t))_{t}-(\boldsymbol{u}_m(s),\boldsymbol{w}_j(s))_{s}|\nonumber \\
\leq & \int_{s}^{t} \left|(\nabla \boldsymbol{u}_m(l),\nabla \boldsymbol{w}_j(l))_{l} \right|\mathrm{d}l+ \int_{s}^{t}\left|b_l\left(\boldsymbol{u}_m\left(l\right),\boldsymbol{u}_m\left(l\right),\boldsymbol{w}_j\left(l\right)\right) \right|\mathrm{d}l+\int_{s}^{t}\left| _{V^*_l}\left<\boldsymbol{f}(l),\boldsymbol{w}_j(l)\right>_{V_l} \right|\mathrm{d}l\nonumber\\
&+\int_{s}^{t}\left|( \boldsymbol{u}_m(l), \boldsymbol{w}_j^{\prime}(l))_l \right|\mathrm{d}l+|X_{j}(t)-X_{j}(s)|\nonumber\\
\leq & C_j\Big(\int_{s}^{t} \|\nabla \boldsymbol{u}_m(l)\|_{l}\mathrm{d}l+\int_{s}^{t}\|\boldsymbol{u}_m(l)\|_{l} \|\nabla \boldsymbol{u}_m(l)\|_{l}\mathrm{d}l+ \int_{s}^{t}\|\boldsymbol{u}_m (l)\|_{l} \mathrm{d}l+\int_{s}^{t}\|\boldsymbol{f}(l)\|^*_{{l}}\mathrm{d}l \Big)\nonumber\\
&+|X_{j}(t)-X_{j}(s)|\nonumber\\
 \leq &C_j\Big((t-s)^{1/2}(1+\sup_{0\leq l\leq T}\|\boldsymbol{u}_m (l)\|_{l})\Big(\int_{0}^{T}\|\nabla \boldsymbol{u}_m(l)\|_{l}^2\mathrm{d} l\Big)^{1/2}+(t-s)\sup_{0\leq l\leq T}\|\boldsymbol{u}_m (l)\|_{l}\nonumber\\
 &\quad\quad +(t-s)^{1/2} \Big(\int_{0}^{T}\|\boldsymbol{f}(l)\|_l^{*^2}\mathrm{d}l\Big)^{1/2}\Big)+Y_j|t-s|^{\frac{p-2}{4p}},
\end{align}
where $C_j$ depends on $M_j$ in (\ref{eq Mj}) and $\sup_{l\in[0,T]}|\mathcal{D}(l)|$, and is independent on $m,s,t$.
For $N>0$ and $q_j>0$, define
\begin{eqnarray*}
&&J_j^N=\Big\{g\in C([0,T];\mathbb{R}); \quad |g(t)-g(s)|\leq
C_j\Big((t-s)^{1/2}\{1+N\}N^{1/2}+(t-s)N\\
&&\quad\quad\quad +(t-s)^{1/2} \left\{\int_{0}^{T}\|\boldsymbol{f}(l)\|_l^{*^2}\mathrm{d}t\right\}^{1/2}\Big)+q_j|t-s|^{\frac{p-2}{4p}}\Big\}.
\end{eqnarray*}
Obviously $J_j^N$ is a class of equicontinuous functions. Now define the relatively compact subset
$K_{N,J}=\bigcap_{i=1}^{\infty}K_{N,J_i^N}$ as in Proposition 4.2 with $J_i$ replaced by $J_i^N$.
\vskip 0.3cm
 Recall (\ref{estimate1}) that
\begin{equation*}
\sup\limits_{m\geq 1}\left(\mathbb{E} \sup\limits_{0\leq t\leq T}\|\boldsymbol{u}_{m}(t)\|_{t}^2+\mathbb{E}\int_{0}^{T} \|\nabla\boldsymbol{u}_{m}(t)\|_{t}^2 \mathrm{d}t\right) \leq M<\infty.
\end{equation*}
If we denote
\begin{align*}
A_{m}^{N}&=\{\omega: \sup\limits_{ t\in[0,T]}\|\boldsymbol{u}_{m}(t,\omega)\|_{t}\leq N\},\\
B_{m}^{N}&=\{\omega: \int_{0}^{T} \|\nabla\boldsymbol{u}_{m}(t,\omega)\|_{t}^2\mathrm{d}t\leq N\},
\end{align*}
then
\begin{align*}
\mathbb{P}\left((A_{m}^{N}\cap B_{m}^{N})^{c}\right)\leq \mathbb{P}\left((A_{m}^{N})^{c}\right)+\mathbb{P}\left((B_{m}^{N})^{c}\right)\leq  M(\frac{1}{N^2}+\frac{1}{N}).
\end{align*}
Given $\varepsilon >0$.  We can choose $N$ sufficiently large such that
\begin{equation}\label{probability}
\mathbb{P}\left((A_{m}^{N}\cap B_{m}^{N})^{c}\right)\leq \frac{\varepsilon}{2}.
\end{equation}

For $i\geq 1$, define $D_i=\{\omega:Y_i(\omega)\leq q_i\}$. We can  take $q_i$ sufficiently large such that
\begin{equation*}
\mathbb{P}(D_i^c)\leq \frac{\mathbb{E}[Y_i^{\frac{8p}{p-2}}]}{q_i^{\frac{8p}{p-2}}}\leq \frac{\varepsilon}{2^{i+1}}.
\end{equation*}
Consequently we have
\begin{equation*}
\mathbb{P}\left(\left(A_m^{N}\cap B_{m}^{N}\cap(\cap_{i\geq 1} D_i)\right)^{c}\right)\leq \frac{\varepsilon}{2}+(\frac{\varepsilon}{2^2}+\cdots+\frac{\varepsilon}{2^{i+1}}+\cdots)\leq \varepsilon.
\end{equation*}
Using the fact that
$$A_m^{N}\cap B_{m}^{N}\cap(\cap_{i\geq 1} D_i)\subset \{\boldsymbol{u}_m\in K_{N,J}\}$$
we deduce that
\begin{equation*}
\mathbb{P}(\boldsymbol{u}_m\in K_{N,J})\geq 1-\varepsilon,
\end{equation*}
proving the tightness.
\end{proof}
\vskip 0.5cm
\noindent{ \bf Proof of Theorem \ref{thm1}}
\vskip 0.3cm
\begin{proof}
We have proved that $\{\mathscr{L}(\boldsymbol{u}_m), m\geq 1\}$ is tight in $\mathbb{L}^2([0,T];H_t)$. There exists a subsequence of $\{\boldsymbol{u}_m, m\geq 1\}$, still denoted by $\{\boldsymbol{u}_m, m\geq 1\}$, such that $\{\mathscr{L}(\boldsymbol{u}_m), m\geq 1\}$ converges  weakly in $\mathbb{L}^2([0,T];H_t)$. By the generalisation of the Skorohod representation theorem (Theorem C.1 in Appendix C in \cite{ZE}), there exists a probability space $({\Omega}^{*}, \mathbb{F}^{*}, \mathbb{P}^{*})$ and a sequence of $\mathbb{L}^2([0,T];H_t)$-valued random variables $\{{\boldsymbol{u}}_m^{*}, m\geq 1\}$ and $\boldsymbol{u}^{*}$, and $C([0,T];\mathbb{R})$-valued random variable $W^*$ such that $\mathbb{P}^{*}\circ(\boldsymbol{u}_m^{*},W^*)^{-1}= \mathbb{P}\circ(\boldsymbol{u}_m,W)^{-1}$ and that
$\boldsymbol{u}^{*}_m\to \boldsymbol{u}^{*}$  in $\mathbb{L}^2([0,T];H_t)$ $\mathbb{P}^{*}$-$a.s.$
By (\ref{estimate1}), we also have
\begin{equation*}
 \mathbb{E}^{\mathbb{P}^*} \sup_{0\leq t\leq T}\|\boldsymbol{u}_m^{*}(t)\|_{t}^2+\mathbb{E}^{\mathbb{P}^*}\int_{0}^{T}\|\boldsymbol{\nabla  u}_m^{*}(t)\|_t^2\mathrm{d} t \leq M.
\end{equation*}Hence,
\begin{equation*}
\boldsymbol{u}^{*}\in \mathbb{L}^2([0,T];V_t)\cap\mathbb{L}^\infty([0,T];H_t),\  \mathbb{P^{*}}\text{-}a.s.
\end{equation*}
and $\boldsymbol{u}_m^{*}\to \boldsymbol{u}^{*}$ weakly in  $\mathbb{L}^2(\Omega\times[0,T]; V_t)$.
 From the equation  (\ref{equ1}) satisfied by  $\widetilde{{\boldsymbol{u}}}_m$ and Remark \ref{Rem 3.2}, we see that for  $\boldsymbol{v}\in \mathscr {C}_{\sigma}(\mathcal{O}_T)= \{\boldsymbol{v}\in C_0^\infty(\mathcal{O}_{T}) | $  div $ \boldsymbol{v} =0$, $\boldsymbol{v} (T)=0$\},
\begin{align*}
&(\mathrm{d}\boldsymbol{{u}}^{*}_{m}(s),\boldsymbol{v}(s))_s+(\nabla \boldsymbol{u}^{*}_m(s), \nabla\boldsymbol{v}(s))_s\mathrm{d}s+(B_s(\boldsymbol{u}^{*}_m(s)),  \boldsymbol{v}(s))_s\mathrm{d}s\\
=&(\boldsymbol{f}(s),\boldsymbol{v}(s))_s\mathrm{d}s+(\boldsymbol{\sigma}^{*}(s),\boldsymbol{v}(s))_s\mathrm{d}W^*(s).
\end{align*}
Intergrating by parts, we have
\begin{align}\label{eq 2}
&-\int_{0}^{T}(\boldsymbol{u}^{*}_{m}(s),\boldsymbol{v}^{\prime}(s))_s\mathrm{d} s +\int_{0}^{T}(\nabla \boldsymbol{u}^{*}_m(s), \nabla\boldsymbol{v}(s))_s\mathrm{d}s+\int_{0}^{T}(B_s(\boldsymbol{u}^{*}_m(s)),  \boldsymbol{v}(s))_s\mathrm{d}s\\=&(\boldsymbol{u}_0^{*},\boldsymbol{v}(0))_0 +\int_{0}^{T}(\boldsymbol{f}(s),\boldsymbol{v}(s))_s\mathrm{d}s+\int_{0}^{T}(\boldsymbol{\sigma}^{*}(s),\boldsymbol{v}(s))_s\mathrm{d}W^*(s).
\nonumber
\end{align}
Since $\boldsymbol{v}\in \mathscr {C}_{\sigma}(\mathcal{O}_T)$ and $\lim\limits_{m\rightarrow\infty}\boldsymbol{u}^{*}_m= \boldsymbol{u}^{*}$  in $\mathbb{L}^2([0,T];H_t)$, $\mathbb{P}^{*}$-a.s.,
\begin{eqnarray*}
&&\Big|\int_{0}^{T}(B_s(\boldsymbol{u}^{*}_m(s)),  \boldsymbol{v}(s))_s\mathrm{d}s
-
 \int_{0}^{T}(B_s(\boldsymbol{u}^{*}(s)),  \boldsymbol{v}(s))_s\mathrm{d}s\Big|\\
&\leq&
\int_{0}^{T}\Big|(B_s(\boldsymbol{u}^{*}_m(s),\boldsymbol{v}(s)),\boldsymbol{u}^{*}_m(s)-\boldsymbol{u}^{*}(s))_s\Big|\mathrm{d}s\\
 &&+
 \int_{0}^{T}\Big|(B_s(\boldsymbol{u}^{*}_m(s)-\boldsymbol{u}^{*}(s),\boldsymbol{v}(s)),\boldsymbol{u}^{*}(s))_s\Big|\mathrm{d}s\\
&\leq&
C_{\boldsymbol{v}}\Big(\int_0^T\|\boldsymbol{u}^{*}_m(s)\|^2_s\mathrm{d}s\Big)^{1/2}
\Big(\int_0^T\|\boldsymbol{u}^{*}_m(s)-\boldsymbol{u}^{*}(s)\|^2_s\mathrm{d}s\Big)^{1/2}\\
&&+
C_{\boldsymbol{v}}\Big(\int_0^T\|\boldsymbol{u}^{*}(s)\|^2_s\mathrm{d}s\Big)^{1/2}
\Big(\int_0^T\|\boldsymbol{u}^{*}_m(s)-\boldsymbol{u}^{*}(s)\|^2_s\mathrm{d}s\Big)^{1/2}
\rightarrow 0,\ \text{as }m\rightarrow \infty,\ \mathbb{P}^{*}\text{-}a.s..
\end{eqnarray*}
Here $C_{\boldsymbol{v}}$ depends on $\sup_{(\boldsymbol{x},s)\in \mathcal{O}_T}\Big(|\frac{\partial \boldsymbol{v}(\boldsymbol{x},s)}{\partial x^1}|+|\frac{\partial \boldsymbol{v}(\boldsymbol{x},s)}{\partial x^2}|\Big)$.

Let $m\to \infty$ in (\ref{eq 2}) to obtain
\begin{align*}
&-\int_{0}^{T}(\boldsymbol{u}^{*}(s),\boldsymbol{v}^{\prime}(s))_s\mathrm{d} s +\int_{0}^{T}(\nabla\boldsymbol{u}^{*}(s), \nabla\boldsymbol{v}(s))_s\mathrm{d}s+\int_{0}^{T}(B_s(\boldsymbol{u}^{*}(s)),  \boldsymbol{v}(s))_s\mathrm{d}s\\
&=(\boldsymbol{u}_0^{*},\boldsymbol{v}(0))_0+\int_{0}^{T}(\boldsymbol{f}(s),\boldsymbol{v}(s))_s\mathrm{d}s+\int_{0}^{T}(\boldsymbol{\sigma}^{*}(s),\boldsymbol{v}(s))_s\mathrm{d}W^*(s).
\end{align*}
This shows that $u^*$ is a martingale solution to the stochastic Navier-Stokes equation (\ref{add}), completing the proof of the existence of a martingale solution.

\vskip 0.2cm
Next we will consider the uniqueness of the solution. Suppose there are two solutions of  (\ref{add}), denoted by $\boldsymbol{u}_1$ and $\boldsymbol{u}_2$, i.e., $\boldsymbol{u}_1$ and $\boldsymbol{u}_2$ satisfy Definition \ref{def solution} with $\boldsymbol{u}$ replaced by $\boldsymbol{u}_1$ and $\boldsymbol{u}_2$, respectively. In particular, we have
\begin{align*}
\mathrm{d}\boldsymbol{u}_1(t)-\Delta \boldsymbol{u}_1(t)\mathrm{d}t+\left(\boldsymbol{u}_1(t)\cdot\nabla\right)\boldsymbol{u}_1(t)\mathrm{d}t+\nabla p(t)\mathrm{d}t&=\boldsymbol{f}(t)\mathrm{d}t+\boldsymbol{\sigma}(t) \mathrm{d} W(t)
\end{align*}
and
\begin{align*}
\mathrm{d}\boldsymbol{u}_2(t)- \Delta \boldsymbol{u}_2(t)\mathrm{d}t+\left(\boldsymbol{u}_2(t)\cdot\nabla\right)\boldsymbol{u}_2(t)\mathrm{d}t+\nabla p(t)\mathrm{d}t&=\boldsymbol{f}(t)\mathrm{d}t+\boldsymbol{\sigma}(t) \mathrm{d} W(t)
\end{align*}
with $\boldsymbol{u}_1(0)=\boldsymbol{u}_2(0)=\boldsymbol{u}_0$.
Setting  $\boldsymbol{z}(t)=\boldsymbol{u}_1(t)-\boldsymbol{u}_2(t)$, then $\boldsymbol{z}(t)$ solves the deterministic equation
\begin{align*}
\partial_t \boldsymbol{z}(t)- \Delta \boldsymbol{z}(t) &=B_t(\boldsymbol{u}_2(t))-B_t(\boldsymbol{u}_1(t)),\\
\boldsymbol{z}(0)&=0.
\end{align*}
Equivalently,
\begin{align}\label{eq z}
\partial_t { \widetilde{\boldsymbol{z}}} (t)+G{ \widetilde{\boldsymbol{z}}} (t) - F { \widetilde{\boldsymbol{z}}}(t) &=\mathcal{N}(\boldsymbol{\widetilde{u}}_2(t))-\mathcal{N}(\boldsymbol{\widetilde{u}}_1(t)),\\
\boldsymbol{\widetilde{z}}(0)&=0.\nonumber
\end{align}
By Lemma \ref{lem 5},
\begin{equation*}
\frac{\mathrm{d}}{\mathrm{d}t}|\widetilde{\boldsymbol{z}}(t)|_t^2=2\left<\widetilde{\boldsymbol{z}}^{\prime}(t)+G\widetilde{\boldsymbol{z}}(t),\widetilde{\boldsymbol{z}}(t)\right>_t,
\end{equation*}
inserting (\ref{eq z}) into the above equation, we get that
\begin{align}\label{uniqueness}
\frac{1}{2} \frac{\mathrm{d}}{\mathrm{d}t}|\widetilde{\boldsymbol{z}}(t)|_t^2+|\nabla_h \widetilde{\boldsymbol{z}}(t)|_t^2 =\left<\mathcal{N}(\boldsymbol{\widetilde{u}}_2(t))-\mathcal{N}(\boldsymbol{\widetilde{u}}_1(t)),\widetilde{\boldsymbol{z}} (t)\right>_t.
\end{align}
For the term on the right, by (\ref{eq b}) and a change of variable, we have the following estimate
\begin{align*}
|\left<\mathcal{N}(\boldsymbol{\widetilde{u}}_1(t))-\mathcal{N}(\boldsymbol{\widetilde{u}}_2(t)),\widetilde{\boldsymbol{z}} (t)\right>_t|&=|b_t(\boldsymbol{z}(t),\boldsymbol{u}_2(t),\boldsymbol{z}(t))|\\
&\leq C_1 |\boldsymbol{\widetilde{z}}(t)|_t |\nabla_h \boldsymbol{\widetilde{z}}(t)|_{t}|\nabla_h \boldsymbol{\widetilde{u}}_2(t)|_{t}\\
&\leq |\nabla_h \widetilde{\boldsymbol{z}}(t)|_t^2+C_7 |\boldsymbol{\widetilde{z}}(t)|_t^2|\nabla_h \boldsymbol{\widetilde{u}}_2(t)|_{t}^2.
\end{align*}
Combined with (\ref{uniqueness}), we have
\begin{align*}
\frac{\mathrm{d}}{\mathrm{d}t}|\widetilde{\boldsymbol{z}}(t)|_t^2\leq 2 C_7 |\boldsymbol{\widetilde{z}}(t)|_t^2|\nabla_h \boldsymbol{\widetilde{u}}_2(t)|_{t}^2.
\end{align*}
Integrate on $[0,t]$ to obatin
\begin{equation*}
|\widetilde{\boldsymbol{z}}(t)|_t^2 \leq \int_{0}^{t} 2 C_7 |\boldsymbol{\widetilde{z}}(s)|_s^2|\nabla_h \boldsymbol{\widetilde{u}}_2(s)|_{s}^2 \mathrm{d} s.
\end{equation*}
Applying Gronwall's lemma, we obtain $\widetilde{\boldsymbol{z}}=0$, proving the pathwise  uniqueness. As the consequence of the Yamada-Watanable theorem, we have proved the existence of a unique strong solution in the probabilistic sense.

The proof of Theorem \ref{thm1} is complete.
\end{proof}
\end{section}
\vskip 0.5cm

\noindent{\bf Acknowledgement}. This work is partially  supported by NSFC (No. 11971456, 11721101), School Start-up Fund (USTC) KY0010000036.

\end{document}